\documentclass[10pt,journal,letterpaper]{IEEEtran}
\usepackage{subfigure}
\usepackage{amssymb, amsmath, amsfonts}
\usepackage{empheq}
\usepackage{stfloats}
\usepackage{caption, cite}
\usepackage{graphicx}
\usepackage{algorithm}
\usepackage{algorithmic}

\usepackage{amsthm}
\usepackage{tikz}
\usepackage{tikzscale}
\usepackage{amsmath}
\allowdisplaybreaks
\usetikzlibrary{shadows,arrows,positioning}
\pgfdeclarelayer{background}
\pgfdeclarelayer{foreground}
\pgfsetlayers{background,main,foreground}

\usetikzlibrary{external}
\usepackage{pgfplots}
\usepackage{cases}

\newtheorem{theorem}{Theorem}

\newtheorem{lemma}{Lemma}

\newtheorem{corollary}{Corollary}
\newtheorem{remark}{Remark}
\newtheorem{assumption}{Assumption}

\newlength\figureheight
\newlength\figurewidth

\DeclareFontFamily{OT1}{pzc}{}
\DeclareFontShape{OT1}{pzc}{m}{it}{<-> s * [1.000] pzcmi7t}{}
\DeclareMathAlphabet{\mathpzc}{OT1}{pzc}{m}{it}

\newcommand{\R}{{\mathbb{R}}}
\newcommand{\E}{{\mathbb{E}}}

\newcommand{\vv}{{\mathbf{v}}}

\newcommand{\bv}{{\mathbf{\bar{v}}}}
\newcommand{\x}{{\mathbf{x}}}
\newcommand{\g}{{\mathbf{g}}}
\newcommand{\bg}{{\mathbf{\bar{g}}}}
\newcommand{\bx}{{\mathbf{\bar{x}}}}
\newcommand{\hx}{{\mathbf{\hat{x}}}}

\newcommand{\lf}{{\nabla{f}}}

\newcommand{\lb}{{\bar{\lambda}}}

\newcommand{\bl}{{\mathbf{L}}}
\newcommand{\bi}{{\mathbf{I}}}
\newcommand{\bk}{{\mathbf{K}}}
\newcommand{\bp}{{\mathbf{P}}}
\newcommand{\bh}{{\mathbf{H}}}

\newcommand\addtag{\refstepcounter{equation}\tag{\theequation}}

\makeatletter

\newcommand{\Rmnum}[1]{\expandafter\@slowromancap\romannumeral #1@}
\makeatother

\title{ \hspace*{\fill} \\\hspace*{\fill} \\ \LARGE \bf{A Communication-Efficient Stochastic Gradient Descent Algorithm for Distributed Nonconvex Optimization}}

\author{Antai Xie, Xinlei Yi, Xiaofan Wang, Ming Cao, and Xiaoqiang Ren
\thanks{A. Xie, X. Wang, and X. Ren are with the School of Mechatronic Engineering and Automation, Shanghai University, Shanghai, China. Emails: \{xatai,\,xfwang,\,xqren\}@shu.edu.cn.}
\thanks{X. Yi is with the Lab for Information \& Decision Systems, Massachusetts Institute of Technology, Cambridge, MA 02139, USA. Email: xinleiyi@mit.edu.}
\thanks{M. Cao is with the Faculty of Science and Engineering, University of Groningen, Groningen, the Netherlands. Email: m.cao@rug.nl.}
}
\begin{document}
	\maketitle
	 \begin{abstract}
 This paper studies distributed nonconvex optimization problems with stochastic gradients for a multi-agent system, in which  each agent aims to minimize the sum of all agents’ cost functions by using local compressed information exchange. We propose a distributed stochastic gradient descent (SGD) algorithm, suitable for a general class of compressors. We show that the proposed algorithm achieves the linear speedup convergence rate $\mathcal{O}(1/\sqrt{nT})$ for smooth nonconvex functions, where $T$ and $n$ are the number of iterations and agents, respectively. If the global cost function additionally satisfies the Polyak--Łojasiewicz condition, the proposed algorithm can linearly converge to a neighborhood of the global optimum, regardless of whether the stochastic gradient is unbiased or not. Numerical experiments are carried out to verify the efficiency of our algorithm.
	\end{abstract}
	
\section{Introduction}

 In recent years, distributed optimization in multi-agent system has become a popular research topic due to its widespread applications in resource allocation \cite{xu2017distributed}, control \cite{nedic2018distributed}, learning \cite{li2020distributed}, and estimation \cite{cattivelli2009diffusion}. The origin of this problem can be traced back to~\cite{tsitsiklis1986distributed,bertsekas1989parallel}. Many effective algorithms have been proposed to address this problem, e.g., distributed (sub)gradient descent~\cite{nedic2009distributed,xu2017convergence,yuan2016convergence}, gradient tracking methods~\cite{qu2017harnessing}, EXTRA~\cite{shi2015extra}, and distributed Newton methods~\cite{varagnolo2015newton,wei2013distributed}. However, these algorithms only considered convex cost functions. In many real-world problems, the cost function is nonconvex, such as empirical risk minimization~\cite{bottou2018optimization} and resource allocation~\cite{tychogiorgos2013non}. To this end, Matei and Baras~\cite{matei2013non} proposed a distributed algorithm for nonconvex constrained optimization utilizing first-order numerical methods. However, this algorithm only converges to a local minimum when the initial values of agents are sufficiently close to that minimum. Tatarrenko and Touri~\cite{tatarenko2017non} studied distributed nonconvex optimization problems on time-varying graphs, proving that the algorithm converges to a local minimum when the cost function has no saddle points. Sun and Hong~\cite{sun2019distributed} introduced a novel nonconvex distributed optimization algorithm using polynomial filtering techniques. Daneshmand \textit{et al.}~\cite{daneshmand2018second} proposed that under the Kurdyka–Łojasiewicz condition, second-order stationary points can be found. Zeng and Yin~\cite{zeng2018nonconvex} presented a gradient descent-based algorithm that converges to first-order stationary points (or their neighborhoods) under diminishing step-sizes (or constant step-sizes).

 In distributed optimization problems, each agent needs to exchange information with its neighbors in order to obtain the global information. However, network bandwidth is typically limited in practical problems. Therefore, it is necessary to consider communication-efficient algorithms. A common solution for agents is to transmit compressed information instead of the raw information. Kajiyama \textit{et al.}~\cite{kajiyama2020linear} achieved linear convergence by combining the gradient tracking algorithm with a compressor having bounded absolute compression errors, and Xiong~\textit{et al.}~\cite{xiong2022quantized} extended the approach in~\cite{kajiyama2020linear} to directed graphs. Liao \textit{et al.}~\cite{liao2022compressed} achieved the same convergence rate with a more general class of compressors. Additionally, the compressed communication algorithms proposed in~\cite{reisizadeh2019robust,taheri2020quantized,yi2022communication} are applicable to nonconvex cost functions.

Most of the aforementioned approaches require the gradient information. However, explicit expressions for gradients are often inaccessible or difficult to obtain. Therefore, it is necessary to consider stochastic gradients. Alistarh \textit{et al.}~\cite{alistarh2017qsgd} and Koloskova \textit{et al.}~\cite{koloskova2019decentralized} proposed communication-efficient stochastic gradient descent (SGD) algorithms by using an unbiased compressor and biased but contractive compressors, respectively. Singh~\textit{et al.}~\cite{singh2022sparq} additionally considered an event-triggered mechanism to further reduce communication costs. Furthermore, \cite{alistarh2017qsgd,koloskova2019decentralized,singh2022sparq,vogels2020practical,koloskova2019decentralized1} achieved an $\mathcal{O}(1/\sqrt{nT})$ convergence rate, where the omitted parameters are not affected by the number of agents $n$. Therefore, they achieved linear speedup\footnote{Linear speedup is achieved if an algorithm uses $n$ times less iterations
than its centralized counterpart to attain the same accuracy.}. However, the authors of~\cite{alistarh2017qsgd,koloskova2019decentralized,singh2022sparq,vogels2020practical,koloskova2019decentralized1} provided analysis only for strongly convex and smooth non-convex cost functions, but did not provide analysis for the Polyak–Łojasiewicz (P--Ł) condition. The P--Ł is weaker than the strong convexity and does not imply the convexity~\cite{yi2021linear}. 
	 
In this paper, we propose a Compressed Primal--dual SGD algorithms (CP-SGD) to solve the distributed nonconvex optimization problem with limited bandwidths. The main contributions of this work are summarized as follows:
 \begin{enumerate}
	\item The proposed algorithm CP-SGD is suitable for a general class of compressors with bounded relative compression errors, which covers the class of compressors used in~\cite{alistarh2017qsgd,koloskova2019decentralized,singh2022sparq,vogels2020practical,koloskova2019decentralized1}. We show that CP-SGD achieves the linear speedup convergence rate $\mathcal{O}(1/\sqrt{nT})$ when the cost functions are smooth (Theorem~\ref{theo:convergence1}). We would like to highlight that, comparing with~\cite{alistarh2017qsgd,koloskova2019decentralized,singh2022sparq,vogels2020practical,koloskova2019decentralized1}, we achieve such linear speedup convergence under weaker assumptions on the (stochastic) gradients.
        \item When the global cost function additionally satisfies the P--Ł condition, we show that CP-SGD linearly converges to a neighborhood of the global minimum with unbiased stochastic gradients~(Theorem~\ref{theo:convergence2}). 
        \item We then consider the biased stochastic gradients. We show that CP-SGD linearly converges to a neighborhood of the global
optimum if the global cost function additionally satisfies the P--Ł condition even for the biased stochastic gradients, but the size of the neighborhood is different from that under unbiased gradients (Theorem~\ref{theo:convergence4}).
	% \item We then establish the linear speedup convergence rate $\mathcal{O}(1/nT)$ for CP-SGD if the parameter of P--Ł condition is known in advance. It is important to note that the mentioned linear speedup convergence rate is match to the optimal for distributed strongly convex optimization problems~\cite{yuan2018optimal,reisizadeh2020fedpaq}. However, CP-SGD does not require convexity of the cost function, and the P--Ł condition is weaker than strong convexity.
 \end{enumerate}

The remainder of this paper is organized as follows. In Section~\ref{sec:Problemsetup}, we introduce the necessary notations and formulate the considered problem. The CP-SGD algorithm is proposed in Section~\ref{sec:Algorithm1}, and its convergence rate without and with P--Ł condition are then analyzed. Some numerical examples are provided in Section~\ref{simulation} to verify the theoretical results. The conclusion and proofs are provided in Section~\ref{conclusion} and Appendix~\ref{app-convergence1}--\ref{app-convergence4}, respectively.
	
\emph{Notations}: $\R$ ($\mathbb{R}_{+}$) is the set of (positive) real numbers. $\mathbb{N}$ the set of positive nature numbers. $\mathbb{R}^n$ is the set of $n$ dimensional vectors with real values. The transpose of a matrix $P$ is denoted by $P^\top$, and we use $P_{ij}$ to denote the element in its $i$-th row and $j$-th column. The Kronecker production is denoted by~$\otimes$. The $n$-dimensional all-one and all-zero column vectors are denoted by $\mathbf{1}_n$ and $\mathbf{0}_n$, respectively. The $n$-dimensional identity matrix is denoted by $I_n$. $diag(x)$ is a diagonal matrix with the vector $x$ on its diagonal. We then introduce two stacked vectors: for a vector $\x\in\R^{nd}$, we denote $\bar{x}=\frac{1}{n}(\mathbf{1}_n^\top\otimes I_d)\x$ and $\mathbf{\bar{\x}}\triangleq\mathbf{1}_n\otimes\bar{x}$. $\vert\cdot\vert$ and $\Vert\cdot\Vert$ denote the absolute value and $l_2$ norm, respectively. For a matrix $W$, we use $\bar{\lambda}_W$ and $\underline{\lambda}_W$ to denote its spectral radius and
minimum positive eigenvalue if the matrix $W$ has positive eigenvalues, respectively. Furthermore, for any square matrix $A$ and vector x with suitable dimension, we denote $\Vert x\Vert_A^2=x^\top Ax$.

\section{Preliminaries and Problem Formulation} \label{sec:Problemsetup}

\subsection{Distributed Optimization}

We consider a network of $n$ agents, where each agent has a private (possibly nonconvex) cost function $f_i:\mathbb{R}^d\mapsto\mathbb{R}$. All agents aim to solve the following optimization problem cooperatively:
\begin{align}\label{P1}
	\min_{x\in\mathbb{R}^d}f(x)=\frac{1}{n}\sum_{i=1}^n f_i(x),
\end{align}
where $x$ is the global decision variable. More specifically, we assume each agent $i$ maintains a local estimate $x_{i,k}\in\mathbb{R}^d$ of $x$ at time step $k$ and use $\lf_i(x_{i,k})$ to denote the gradient of $f_i$ with respect to $x_{i,k}$. Furthermore, we assume that each agent in the network only has access to the stochastic gradient of its local cost function. We use $\tilde{\nabla}f_{i,k}=g_i(x_{i,k},\xi_{i,k})$ to denote the stochastic gradient of $f_i$ at $x_{i,k}$ with $\xi_{i,k}$ being a random variable.
\subsection{Graph Theory}
	
In this paper, we use the undirected graph $\mathcal{G}(\mathcal{V}, \mathcal{E})$ to denote the communication network with $n$ agents, where $\mathcal{V}=\{1, 2, \ldots, n\}$ is the set of the agents' indices and $\mathcal{E} \subseteq \mathcal{V} \times \mathcal{V}$ is the set of edges. The edge $(i,j)\in\mathcal{E}$ if and only if agents~$i$ and~$j$ can communicate with each other. The coupling weight matrix of $\mathcal{G}$ is denoted by $W=[w_{ij}]_{n\times n}\in\mathbb{R}^{n\times n}$ with $w_{ij}>0$ if $(i,j)\in\mathcal{E}$, and $w_{ij}=0$, otherwise. Furthermore, the neighbor agent set of agent $i$ is denoted by $\mathcal{N}_i=\{j\in\mathcal{V}|~(i,j)\in\mathcal{E}\}$. The degree matrix is denoted as $D=diag[d_1,d_2,\cdots,d_n]$, where $d_i=\sum_{j}^n w_{ij},~\forall i\in\mathcal{V}$. The Laplacian matrix of graph $\mathcal{G}$ is denoted by $L=D-W$.

\subsection{Assumptions}

In this subsection, we introduce the following assumptions on graph, local cost functions $f_i$ and stochastic gradient $g_i(x_{i,k},\xi_{i,k})$.
\begin{assumption}\label{as:strongconnected}
	The undirected graph $\mathcal{G}(\mathcal{V}, \mathcal{E})$ is connected.
\end{assumption}

\begin{assumption}\label{as:smooth}
	Each local cost function $f_i$ is $L_f$-smooth, for some $L_f>0$, namely for any $x,y\in\mathbb{R}^d$,
	\begin{align}\label{eqn:smooth}
	\left\Vert \lf_i(x)-\lf_i(y)\right\Vert\leq L_f\left\Vert x-y\right\Vert.
	\end{align}
\end{assumption}
From~\eqref{eqn:smooth}, we have
\begin{align}\label{eqn:smooth1}
\vert f_i(y)-f_i(x)-(y-x)^\top\lf_i(x)\vert\leq\frac{L_f}{2}\left\Vert y-x\right\Vert^2.
\end{align}

Assumption~\ref{as:strongconnected} and~\ref{as:smooth} are standard for distributed optimization problems and widely used in existing works, e.g.,~\cite{shi2015extra,yang2019survey,yi2022communication}.  
\begin{assumption}\label{as:independent}
    The random variables $\{\xi_{i,k},i\in\mathcal{V},k\in\mathbb{N}\}$ are independent of each other.
\end{assumption}
\begin{assumption}\label{as:unbiase}
    The stochastic gradient $g_i(x,\xi_{i,k})$ is unbiased, that is,
    \begin{align}
        \E_{\xi_{i,k}}[g_i(x,\xi_{i,k})]=\lf_i(x),~\forall i\in\mathcal{V},~k\in\mathbb{N}, x\in\R^,
    \end{align}
    where $\mathbb{E}_{\xi_{i,k}}$ denotes the expectation with respect to $\xi_{i,k}$.
\end{assumption}
\begin{assumption}\label{as:boundedvar}
 There exists a constant $\sigma>0$ such that 
    \begin{align}
    \E_{\xi_{i,k}}\Vert g_i(x,\xi_{i,k})-\lf_i(x)\Vert^2\leq\sigma^2,~\forall i\in\mathcal{V},~k\in\mathbb{N}, x\in\R^d.
    \end{align}
\end{assumption}
\begin{remark}
    Assumptions~\ref{as:independent}--\ref{as:boundedvar} are commonly used for stochastic gradients, e.g.,~\cite{alistarh2017qsgd,singh2022sparq,huang2023cedas}. Furthermore, Assumption~\ref{as:boundedvar}~only requires that the random gradient has a bounded variance, which is weaker than the bounded second moment or the bounded gradient used in~\cite{koloskova2019decentralized1,stich2018local}.
\end{remark}

We then make the following assumptions on the global cost function~$f$.

\begin{assumption}\label{as:finite}
Let $f^*$ be the minimum function value of the problem~\eqref{P1}. We assume $f^*>-\infty$.
\end{assumption}
\begin{assumption}\label{as:PLcondition}
(Polyak–Łojasiewicz (P--Ł) condition~\cite{yi2022communication}) There exists a constant $\nu>0$ such that for any $x\in\mathbb{R}^d$,
\begin{align}
\frac{1}{2}\left\Vert \lf(x)\right\Vert^2\geq \nu(f(x)-f^*).
\end{align}
\end{assumption}
\begin{remark}
    Note that the P--Ł condition does not imply the convexity of the global cost function $f$, and is weaker than strong convexity~\cite{yi2021linear}. Furthermore, it is easy to check that all stationary points of~\eqref{P1} under P--Ł condition are the global minimizer.
\end{remark}
% \begin{remark}
% Assumptions~\ref{as:smooth}--\ref{as:finite} are standard in distributed nonconvex optimization, e.g., \cite{shi2015extra,yang2019survey,yi2022communication}. Furthermore, as point out in~\cite{yi2022communication}, the P--Ł condition is a weaker assumption than strong convexity and ensures that each stationary point of problem~\eqref{P1} is a global minimizer.
% \end{remark}

\subsection{Compression Method}
To save communication resources, we assume that agents in the network only exchange compressed information. More specifically, for any $x\in\R^d$, we consider a general class of stochastic compressors $C(x)$ that satisfy the following assumption.

\begin{assumption}\label{as:compressor}
For some constants $\varphi\in(0,1]$ and $r>0$ the compressor $C(\cdot):\mathbb{R}^d\mapsto\mathbb{R}^d$ satisfies 
\begin{align*}
    \mathbb{E}_C\left[
        \left\Vert \frac{C(x)}{r}-x\right\Vert^2\right]\leq(1-\varphi)\left\Vert x\right\Vert^2, \forall x\in\mathbb{R}^d,\addtag\label{eq:propertyofcompressors}
\end{align*}
where $\mathbb{E}_C$ denotes the expectation with respect to the stochastic compression operator $C$. 
\end{assumption}
From~\eqref{eq:propertyofcompressors} and the Cauchy--Schwarz inequality, one obtains that
\begin{align}\label{eq:propertyofcompressors1}
    \mathbb{E}_C\left[
    \left\Vert C(x)-x\right\Vert^2\right]\leq r_0\left\Vert x\right\Vert^2, \forall x\in\mathbb{R}^d.
\end{align}
where $r_0=2r^2(1-\varphi)+2(1-r)^2$.

 As pointed out in~\cite{liao2022compressed} that compressors under Assumption~\ref{as:compressor} cover the class of compressors used in\cite{alistarh2017qsgd,koloskova2019decentralized,singh2022sparq,vogels2020practical}. Furthermore, it is easy to verify that the following commonly used compressors satisfy Assumption~\ref{as:compressor}. 
\begin{itemize}
	\item Greedy (Top-$k$) quantizer~\cite{beznosikov2023biased}:
	\begin{align*}
		C_1(x):=\sum_{i_s=1}^kx_{(i_s)}e_{i_s},\addtag\label{eq:compressor1}
	\end{align*}
where $x_{(i_s)}$ is the $i_s$-th coordinate of $x$ with $i_1,\dots,i_k$ being the indices of the largest $k$ coordinates in magnitude of $x$, and $e_1,\dots,e_d$ are the standard unit basis vectors in $\R^d$.
    \item Biased $b$-bits quantizer~\cite{koloskova2019decentralized}:
    \begin{align*}
    	C_2(x):=\frac{\Vert x\Vert}{\xi}\cdot \text{sign}(x)\cdot 2^{-(b-1)}\circ \left\lfloor\frac{2^{(b-1)}\vert x\vert}{\Vert x\Vert} +u\right\rfloor,\addtag\label{eq:compressor2}
    \end{align*}
where $\xi=1+\min\{\frac{d}{2^{2(b-1)}},\frac{\sqrt{d}}{2^{(b-1)}}\}$, $u$ is a random dithering vector uniformly sampled from $[0,1]^d$, $\circ$ is the Hadamard product, and $\text{sign}(\cdot)$, $|\cdot|$, $\lfloor\cdot\rfloor$ are the element-wise sign, absolute and floor functions, respectively. 
\end{itemize}

\section{Compressed Primal--Dual SGD Algorithm }\label{sec:Algorithm1}
In this section, we propose a Compressed Primal--dual SGD algorithm (CP-SGD) to solve the problem~\eqref{P1} under the limited bandwidths. Furthermore, we analyze the convergence rate of CP-SGD without and with the P--Ł condition.
\subsection{Algorithm Description}
To solve the distributed nonconvex optimization problem~\eqref{P1}, Yi~\textit{et~al.}~\cite{yi2022primal} proposed the following distributed primal--dual SGD algorithm
\begin{align*}
&~x_{i,k+1}=x_{i,k}-\eta_k(\gamma_k\sum_{j=1}^n L_{ij}x_{j,k}+\omega_k  v_{i,k}+\tilde{\nabla} f_{i,k}),\addtag\label{iterationxp}\\
&~v_{i,k+1}=v_{i,k}+\eta_k\omega_k  \sum_{j=1}^n L_{ij}x_{j,k},\addtag\label{iterationv}
\end{align*}
where $\eta_k$ is step-size, $\gamma_k$ as well as$~\omega_k $ are time-varying positive parameters, and $v_{i,k}$ is the auxiliary variable of agent~$i$. 

To accommodate limited bandwidth, each agent~$i$ needs to transmit the compressed information $C(x_{j,k}-x_{j,k}^c)$ to its neighbors. We then use an estimated variable $\hat{x}_{i,k}$ to reduce the compressed error. Specifically, the updates for agent $i\in\mathcal{V}$ can be described as follows:
\begin{align*}
&~x_{i,k+1}=x_{i,k}-\eta_k(\gamma_k \sum_{j=1}^n L_{ij}\hat{x}_{j,k}+\omega_k  v_{i,k}+\tilde{\nabla} f_{i,k}),\addtag\label{eq:iterationxp2}\\
&~v_{i,k+1}=v_{i,k}+\eta_k\omega_k \sum_{j=1}^n L_{ij}\hat{x}_{j,k},\addtag\label{eq:iterationv2}
\end{align*}
where 
\begin{align*}
	&\hat{x}_{j,k}=x_{j,k}^c+C(x_{j,k}-x_{j,k}^c),\addtag\label{citerationx}\\
	&x_{j,k+1}^c=(1-\alpha_x)x_{j,k}^c+\alpha_x\hat{x}_{j,k},\addtag\label{citerationxc}
\end{align*}
with $\alpha_x$ being a positive parameter and initial compressed information $x_{i,0}^c=\mathbf{0}_d$,~$\forall i\in\mathcal{V}$. We then describe the CP-SGD in Algorithm~\ref{Al:CP-SGD}.

\begin{algorithm}[]
	\caption{CP-SGD Algorithm}%算法标题
	\label{Al:CP-SGD}
	\begin{algorithmic}[1]%一行一个标行号
		\STATE \textbf{Input:} Stopping time $T$, adjacency matrix $W$, and positive parameters $\{\eta_k\}$, $\{\gamma_k\}$, $\{\omega_k\}$, $\alpha_x$.
		\STATE \textbf{Initialization:} Each ~$i\in\mathcal{V}$ chooses arbitrarily $x_i(0)\in\mathbb{R}^d$, $x^c_i(0)=\bf{0}_d$, $v_i(0)=\bf{0}_d$.
		\FOR{$k=0,1,\dots,T-1$}
		\FOR {for $i\in\mathcal{V}$ in parallel} 
		\STATE Compute $C(x_{i,k}-x_{i,k}^c)$ and broadcast it to its neighbors $\mathcal{N}_i$.
		\STATE Receive $C(x_{j,k}-x_{j,k}^c)$ from $j\in\mathcal{N}_i$.	
		\STATE Update $x_{i,k+1}$ and $v_{i,k+1}$ according to~\eqref{eq:iterationxp2} and~\eqref{eq:iterationv2}, respectively.
		\STATE Update $x_{j,k+1}^c$ from~\eqref{citerationxc}.
		\ENDFOR
		\ENDFOR
		\STATE \textbf{Output:} \{$x_{i,k}$\}.
	\end{algorithmic}
\end{algorithm}

\subsection{Convergence Analysis of CP-SGD}
In this section, we first show the convergence of CP-SGD  for smooth nonconvex cost functions.

\begin{theorem}\label{theo:convergence1}
	Suppose Assumptions~\ref{as:strongconnected}--\ref{as:boundedvar} and \ref{as:compressor} hold and in Algorithm~\ref{Al:CP-SGD}, let~$\gamma_k=\beta_1\omega_k,
 ~\eta_k=\frac{\beta_2}{\omega_k},~\omega_k=\omega>\beta_3$, and $\alpha_x\in(0,\frac{1}{r})$, $\forall k\in\mathbb{N}$ where $~\beta_1>c_0,~\beta_2>0$ with $c_0,\beta_3$ are positive constants given in Appendix~\ref{app-convergence1}. Then, for any $T\in\mathbb{N}$, we have
 \begin{align*}
    &\frac{1}{T}\sum_{k=0}^{T-1}\mathbb{E}\left[\frac{1}{n}\sum_{i=1}^n\Vert x_{i,k}-\bar{x}_k\Vert^2\right]\leq \mathcal{O}(\frac{1}{T})+\mathcal{O}(\frac{1}{\omega^2}),\addtag\label{eq:theo11}\\
    &\frac{1}{T}\sum_{k=0}^{T-1}\mathbb{E}\Vert \nabla f(\bar{x}_k)\Vert^2\leq\mathcal{O}(\frac{\omega}{T})+\mathcal{O}(\frac{1}{n\omega})+\mathcal{O}(\frac{1}{T})+\mathcal{O}(\frac{1}{\omega^2}).\addtag\label{eq:theo12}
 \end{align*}
\end{theorem}
\begin{proof}
	See Appendix~\ref{app-convergence1}.
\end{proof}
From the right-hand side of~\eqref{eq:theo12}, it is easy to see that the linear speedup convergence rate can be achieved if  $\omega=\sqrt{T}/\sqrt{n}$, which is presented in the following result.
\begin{corollary}\label{coro1}
    Under the same assumptions and parameters settings in Theorem~\ref{theo:convergence1}, let $\omega=\beta_2\sqrt{T}/\sqrt{n}$, for any $T>n(\beta_3/\beta_2)^2$, then
\begin{align*}
    &\frac{1}{T}\sum_{k=0}^{T-1}\mathbb{E}\left[\frac{1}{n}\sum_{i=1}^n\Vert x_{i,k}-\bar{x}_k\Vert^2\right]\leq \mathcal{O}(\frac{n}{T}),\addtag\label{eq:coro11}\\
    &\frac{1}{T}\sum_{k=0}^{T-1}\mathbb{E}\Vert \nabla f(\bar{x}_k)\Vert^2\leq\mathcal{O}(\frac{1}{\sqrt{nT}})+\mathcal{O}(\frac{n}{T}).\addtag\label{eq:coro12}
 \end{align*}
\end{corollary}
\begin{remark}
    The omitted parameters in $\mathcal{O}(\frac{1}{\sqrt{nT}})$ in~\eqref{eq:coro12} is unaffected by any parameters related to communication graphs. In other words, CP-SGD is suitable for any connected graph. Furthermore, it is important to note that, although the same convergence rate is established in references~\cite{alistarh2017qsgd,koloskova2019decentralized,singh2022sparq,vogels2020practical,koloskova2019decentralized1}, they require additional assumptions. Specifically, the methods~\cite{alistarh2017qsgd,koloskova2019decentralized,singh2022sparq,koloskova2019decentralized1} required the stochastic gradients have second bounded moment and the method~~\cite{ vogels2020practical} assumed that $\frac{1}{n}\|\lf_i(x)-\lf(x)\|^2$ is uniformly bounded. 
\end{remark}

Then we provide the linear convergence of CP-SGD with the P--Ł condition.
\begin{theorem}\label{theo:convergence2}
	Suppose Assumptions~\ref{as:strongconnected}--\ref{as:compressor} hold and in Algorithm~\ref{Al:CP-SGD}, let~$\gamma_k=\beta_1\omega_k,~\omega_k=\omega>\beta_3$, $\alpha_x\in(0,\frac{1}{r})$, and $\eta_k=\frac{\beta_2}{\omega_k}$, $\forall k\in\mathbb{N}$ where $~\beta_1>c_0,~\beta_2>0$. Then we have
 \begin{align*}
     \E\bigg[\sum_{i=1}^n\Vert x_{i,k}-\bar{x}_k\Vert^2+&n(f(\bar{x}_{k})-f^*)\bigg]\\
     &\leq(1-\eta\bar{\beta})^{k+1}\frac{V_0}{\check{c}_1}+\frac{\check{c}_2\sigma^2\eta}{\bar{\beta}\check{c}_1},\addtag\label{eq:linearconverge}
 \end{align*}
where $\bar{\beta},\beta_3,c_0,\check{c}_1,\check{c}_2,V_0$ are positive constants given in Appendices~\ref{app-convergence1} and~\ref{app-convergence2} with $0<\eta\bar{\beta}\leq\frac{2}{3}$.
\end{theorem}
\begin{proof}
	See Appendix~\ref{app-convergence2}.
\end{proof} 
\begin{remark}
    Since $0<\eta\bar{\beta}\leq\frac{2}{3}$, Theorem~\ref{theo:convergence2} shows that CP-SGD can linearly converge to a neighborhood of the global optimum. It can be seen from the second term on the right-hand of~\eqref{eq:linearconverge}, the inaccuracy of convergence is caused by the variance of the stochastic gradients. In other words, CP-SGD can linearly converge to the optimum when gradients are available.
\end{remark}

%  \begin{theorem}\label{theo:convergence3}
%  	Suppose Assumptions~\ref{as:strongconnected}--\ref{as:compressor} hold, and P--Ł constant $\nu$ is known in advance. In Algorithm~\ref{Al:CP-SGD}, let~$\gamma_k=\beta_1\omega_k,~\omega_k=\beta_0(k+t_1)$, $\alpha_x\in(0,\frac{1}{r})$, and $\eta_k=\frac{\beta_2}{\omega_k}$, $\forall k\in\mathbb{N}$ where $\beta_0\in[\tilde{c}\nu\beta_2/4,\nu\beta_2/4)$, $\beta_1>\bar{c}_1$, $0<\beta_2<\bar{c}_2$, $t_1>\bar{c}_5$ with $1>\tilde{c}>0$ being a constant, $\bar{c}_1$, $\bar{c}_2$, and $\bar{c}_5$ are constants given in Appendix~\ref{app-convergence3}. Then, for any $T\in\mathbb{N}$, we have
%  \begin{align*}
%      &\E[\Vert\x_k-\bx_k\Vert^2]=\mathcal{O}(\frac{1}{T^2}),\addtag\label{eq:convergenceofth31}\\
%      &\E[f(\bar{x}_{k})-f^*)]=\mathcal{O}(\frac{1}{T^2})+\mathcal{O}(\frac{1}{nT}).\addtag\label{eq:convergenceofth32}
%  \end{align*}
 
% \end{theorem}
%  \begin{remark}
%      As shown in~\cite{rakhlin2011making}, $\mathcal{O}(1/T)$ is the optimal convergence rate for centralized strongly convex optimization problems. The same linear speedup convergence rate has established in~\cite{reisizadeh2020fedpaq,lan2020communication} and~\cite{yi2022primal}. However, the former assumes convexity of the cost function, while the latter does not consider compressed communication.
 % \end{remark}
 We then consider the convergence of CP-SGD with biased stochastic gradients.
 \begin{theorem}\label{theo:convergence4}
	Suppose Assumptions~\ref{as:strongconnected}--\ref{as:independent} and~\ref{as:boundedvar}--\ref{as:compressor} and in Algorithm~\ref{Al:CP-SGD}, let~$\gamma_k=\beta_1\omega_k,~\omega_k=\omega>\check{\beta}_3$, $\alpha_x\in(0,\frac{1}{r})$, and $\eta_k=\frac{\beta_2}{\omega_k}$, $\forall k\in\mathbb{N}$ where $~\beta_1>c_0,~\beta_2>0$. Then we have
 \begin{align*}
     \E[\sum_{i=1}^n\Vert x_{i,k}-\bar{x}_k\Vert^2+&n(f(\bar{x}_{k})-f^*)]\\
     &\leq(1-\eta\bar{\beta})^{k+1}\frac{V_0}{\check{c}_1}+\frac{\check{c}_3n\sigma^2}{\bar{\beta}\check{c}_1},\addtag\label{eq:linearconverge1}
 \end{align*}
where $\bar{\beta},\check{\beta}_3,c_0,\check{c}_1,\check{c}_3,V_0$ are positive constants given in Appendices~\ref{app-convergence1},~\ref{app-convergence2}, and~\ref{app-convergence4} with $0<\eta\bar{\beta}\leq\frac{2}{3}$.
\end{theorem}
\begin{proof}
	See Appendix~\ref{app-convergence4}.
\end{proof} 
\begin{remark}
From~\eqref{eq:linearconverge1}, we know that CP-SGD linearly converges to a neighborhood of the global
optimum even without the unbiased assumption. Furthermore, compared~\eqref{eq:linearconverge} and~\eqref{eq:linearconverge1}, it can be seen that the size of neighbourhood is different. More specifically, since $\check{c}_2=\mathcal{O}(n)$, the second term of right-hand side of~\eqref{eq:linearconverge} and~\eqref{eq:linearconverge1} are in an order of $\mathcal{O}(n\eta)$ and~$\mathcal{O}(n)$, respectively. In other words, the former achieves better accuracy than the latter when the step-size $\eta$ is sufficiently small.
\end{remark}

\section{simulation}\label{simulation}
In this section, simulations are given to verify the validity of CP-SGD. We consider a distributed estimation problem with $n=6$ agents and they communicate on a connected undirected graph, whose topology is shown in Fig.~\ref{Fig:graph}. Specifically, we assume agent $i$ aims to solve the following nonconvex distributed binary classification problem~\cite{antoniadis2011penalized,yi2021linear,sun2019distributed}
\begin{align*}
	&~~~~~\min_{x}f(x)=\frac{1}{6}\sum_{i=1}^6 f_i,\\\addtag\label{eq:simulation0}
 &f_i\left(x_i\right)=\frac{1}{m} \sum_{j=1}^m \log \left(1+\exp \left(-u_{i j} x_i^{\top} v_{i j}\right)\right)+\sum_{s=1}^d \frac{\lambda \alpha x_{i, s}^2}{1+\alpha x_{i, s}^2},
\end{align*}
where $v_{ij}\in\R^d$ is the feature vector and randomly generated with standard Gaussian distribution $\mathcal{N}(0,1)$, $u_{ij}\in\{-1,1\}$ is the label and randomly generated with uniformly distributed pseudorandom integers taking the values $\{-1,1\}$ and $x_{i, s}$ is the $s$-th coordinate of $x_i$. Specifically, we assume $\lambda=0.001,\alpha=1,m=200$ and the initial value of each agent $x_i(0)$ is randomly chosen in $[0,1]^{10}$. Furthermore, we assume each agent~$i$ has access to a noisy gradient $g^s_{i,k}(x_i)=\lf_i(x_i)+\delta_s$, where $\delta_s\sim \mathcal{N}(0,0.5)$ is the noise.

We consider the compressors $C_1$ and $C_2$ in \eqref{eq:compressor1} and~\eqref{eq:compressor2}, respectively, with $k=2$ and $b=2$ in the following simulations. We compare our algorithm with time-varying parameters and constant parameters (CP-SGD-T and CP-SGD-F) with the distributed SGD algorithm (DSGD) and compressed algorithm~\cite{koloskova2019decentralized1} (Choco-SGD) under different parameters as specified in TABLE~\ref{tab:parameter}. We use the residual $R_k\triangleq\min_{t\leq k}\Vert\x_t-\x^*\Vert^2$ to evaluate the convergence rate.

 \begin{table}[h]
	\centering
	\begin{tabular}{lc c c c c}\hline
		Algorithm & Compressor & $\gamma_k$ & $\omega_k$ & $\eta_k$ & $\alpha_x$ \\\hline
		DSGD &---&---&---&0.05&---\\
		Choco-SGD-C1 &$C_1$&0.2&---&0.05&---\\
		CP-SGD-F-C1 &$C_1$&4&0.5&0.05&0.2\\
		CP-SGD-F-C2  &$C_2$&4&0.5&0.05&0.2\\
  	CP-SGD-T-C1  &$C_1$&45k&5k&$10^{-4}/k$&0.2\\\hline
	\end{tabular}
	\caption{Parameter setting for different algorithms.}
	\label{tab:parameter}
\end{table}

\begin{center}
	\begin{figure}
		\center
		\usetikzlibrary{positioning, fit, calc}

\begin{tikzpicture}[->,>=stealth',shorten >=1pt,auto,node distance=2cm,
thick,main node/.style={circle,fill=gray!20,draw,font=\sffamily\small\bfseries}]

\node[main node] (1) {1};
\node[main node] (2) [ right of=1] {2};
\node[main node] (3) [ below right of=2] {3};
\node[main node] (4) [ below left of=3] {4};
\node[main node] (5) [ left of=4] {5};
\node[main node] (6) [ below left of=1] {6};

%	\node[main node] (7) [ left of=13] {7};
%	\node[main node] (14) [ below left of=11] {14}; 
%	\node[main node] (10) [ below left of=14] {10};
%	\node[main node] (5) [below of=10]{5};
%	\node[main node] (9) [ right of=11] {9}; 
%	\node[main node] (2)  [below right of=14] {2};
%	\node[main node] (3) [ right of=2] {3};
%	\node[main node] (4) [below left of=2] {4};
%	
%	\node[main node] (15) [ below right of=2] {15};
%	\node[main node] (16) [right of=3] {16};
%	\node[main node] (6) [ below of=16] {6};
%	

\path[every node/.style={font=\sffamily\small}]
(1) edge  node[] {} (4)
edge  node[left] {} (2)	  
edge   node[above]{} (6)
(2) edge  node[left] {} (5)
edge  node[above]{} (1)
edge  node[left] {} (3)
(3) edge  node[left] {} (4)
edge  node[left] {} (2)
(4) edge  node[above]{} (3)
edge  node[above]{} (1)	
edge  node[left] {} (5)
(5) edge  node[left] {} (2)
edge node[left] {} (4)
edge node[left] {} (6)
(6) edge node[above]{} (5)
edge node[above]{} (1);
%  edge [bend right] node[left] {0.3} (2)
%  edge [loop above] node {0.1} (1)
%	(2) edge node [right] {} (11)
% edge node {0.3} (4)
%  edge [loop left] node {0.4} (2)
% edge [bend right] node[left] {0.1} (3)
%	(3) edge node [right] {} (6)
%	edge node [right] {} (11)
%	edge node [right] {} (16)
% edge [bend right] node[right] {0.2} (4)
%	(4) edge node [left] {} (2)
%	edge node [left] {} (10)
%	edge node [left] {} (11)
%	edge  [bend right] node{} (13)
%	edge node [bend right]{} (15)
%	(6)  edge[bend left] node [left] {} (5)
%	edge node [left] {} (15)
%	(8)  edge node [left] {} (11)
%	edge node [left] {} (2)
%	(9)  edge node [left] {} (11)
%	(10)edge node[left]{}  (2)
%	edge[bend left] node [left] {} (11)
%	(13) edge node [left] {} (11)
%	edge node [left] {} (12)
%	(14)  edge node [left] {} (2)
%	edge node [left] {} (4)
%	edge node [left] {} (8)
%	edge node [left] {} (10)
%	edge node [left] {} (11)
%	edge node [left] {} (13);
%	 edge [loop right] node {0.6} (4)
\end{tikzpicture}
		\caption{A connected undirected graph consisting of 6 agents.}
		\label{Fig:graph}
	\end{figure}
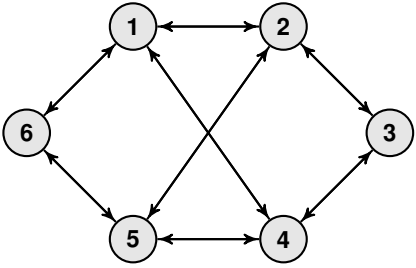
\end{center}
\begin{figure}
	\centering
	\includegraphics[width=1\linewidth]{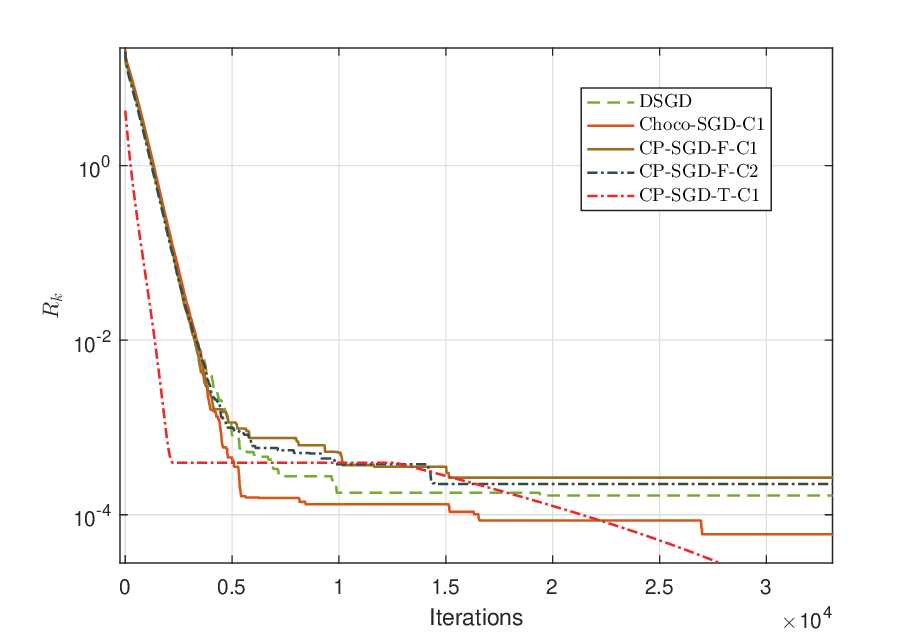}
	\caption{The evolution of residual under DSGD, Choco-SGD, and CP-SGD.}
	\label{fig:convergence1}
\end{figure}
\begin{figure}
	\centering
	\includegraphics[width=1\linewidth]{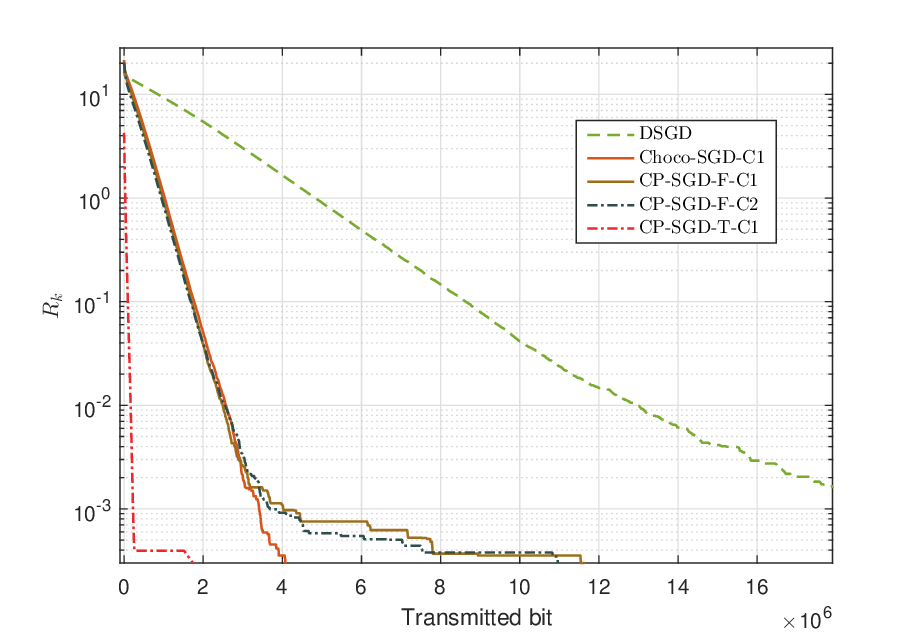}
	\caption{The evolution of residual with respect to the transmitted bits under DSGD, Choco-SGD, and CP-SGD.}
	\label{fig:bit1}
\end{figure}

Fig.~\ref{fig:convergence1} shows that $\x(k)$ converges to the optimal point $\x^*$ under CP-SGD with different step-sizes and compressors. Since residue $R_k$ represents the minimum value in the time interval $[0,k]$, the descent will be step-like as shown in~Fig.~\ref{fig:convergence1}. The convergence rate of CP-SGD-T and CP-SGD-F are faster than and close to Choco-SGD~\cite{koloskova2019decentralized1} under same compressor, respectively. Furthermore, Fig.~\ref{fig:bit1} illustrates that compared to other algorithms, CP-SGD-T converges to the same accuracy with fewer bits. This also demonstrates the efficiency of CP-SGD-T.

\section{conclusion}\label{conclusion}
In this paper, we investigated distributed nonconvex optimization under limited communication. Specifically, we proposed a compressed primal--dual SGD algorithm. For a general class of compressors with bounded relative compression errors, we established the linear speedup convergence rate $\mathcal{O}(1/\sqrt{nT})$ for smooth cost functions where $T$ and $n$ are the number of iterations and agents, respectively. Note that the convergence rate does not depend on any parameters that are related to communication graph. If the global cost function additionally satisfies the Polyak-Łojasiewicz condition, we proved that the proposed algorithm linearly converges to a neighborhood of the global optimum. Furthermore, we also proved that, even for the biased stochastic gradients, the proposed algorithm still linearly converges to a neighborhood of the global optimum but with a different neighborhood. Future work includes extending the study to directed graphs and considering the privacy issues.

\appendices
\section{Supporting Lemmas}
We first introduce some useful vector and matrix inequalities.
\begin{lemma}\label{lemma:usefulinequalities}
For $u,v\in\R^d$, and $\forall s>0$ we have
\begin{align*}
	&~~~~~~u^\top v\leq \frac{s}{2}\Vert u\Vert^2+\frac{1}{2s}\Vert v\Vert^2,\addtag\label{eq:inequality1}\\
	&\Vert u+v\Vert^2\leq(1+s)\Vert u\Vert^2+(1+\frac{1}{s})\Vert v\Vert^2.\addtag\label{eq:inequality2}
\end{align*}
\end{lemma}
\begin{lemma}\label{lemma:lsmooth}
	\cite{liao2022compressed} Suppose the function $f(x): \R^d\mapsto\R$ is smooth with constant $L_f>0$, we have
 \begin{align*}
     \|\lf(x)\|^2\leq2L_f(f(x)-f^*),\addtag\label{eq:lsmooth}.
 \end{align*}
	\end{lemma}

\begin{lemma}
~(Lemma 2 in~\cite{yi2022communication}) Suppose Assumption~\ref{as:strongconnected} holds, let $L$ be the Laplacian matrix of the graph $G$ and $K_n=\mathbf{I}_n-\frac{1}{n}\mathbf{1}_n\mathbf{1}_n^\top$. Then $L$ and $K_n$ are positive semi-definite, $L\leq\bar{\lambda}_L\mathbf{I}_n$, $\bar{\lambda}_{K_n}=1$, 
\begin{align}
&K_nL=LK_n=L,\label{eq:propertyofk}\\
&0\leq\underline{\lambda}_LK_n\leq L\leq\bar{\lambda}_LK_n.\label{eq:propertyofk1}
\end{align}
Moreover, there exists an orthogonal matrix $[r~R]\in\R^{n\times n}$ with $r=\frac{1}{\sqrt{n}}\mathbf{1}_n$ and $R\in\R^{n\times(n-1)}$ such that
\begin{align}
	&PL=LP=K_n,\label{eq:propertyofp}\\
	&\bar{\lambda}_L^{-1}\mathbf{I}_n\leq P\leq\underline{\lambda}_L^{-1}\mathbf{I}_n,\label{eq:propertyofp1}
\end{align}
where 
\begin{align*}
	P=\begin{bmatrix}
		r&R
	\end{bmatrix}
	\begin{bmatrix}
		\lambda_n^{-1}&0\\
		0&\Lambda_1^{-1}
	\end{bmatrix}
	\begin{bmatrix}
		r^\top\\
		R^\top\top
	\end{bmatrix},
\end{align*}
with $\Lambda_1=\text{diag}([\lambda_2,\dots,\lambda_n])$ and $0\leq\lambda_2\leq\cdots\leq\lambda_n$ being the nonzero eigenvalues of $L$.
\end{lemma}

\section{The proof of Theorem~\ref{theo:convergence1}}\label{app-convergence1}
\subsection{Notations and useful lemma}
Denote $\tilde{f}(\x_{k})=\sum_{i=1}^nf_i(x_{i,k}),~\bl=L\otimes \bi_d$, $\bp=P\otimes\bi_d$, $\g_k=\nabla\tilde{f}(\x_{k})$, $\bg_k=\bh\g_k$, $\g_k^b=\nabla\tilde{f}(\bx_{k})$, $\bg_k^b=\bh\g_k^b=\mathbf{1}_n\otimes\nabla f(\bar{x}_k)$. We first construct some auxiliary functions and provide the following lemma
\begin{lemma}\label{lemma:lyapunov}
Suppose Assumptions~\ref{as:strongconnected}--\ref{as:boundedvar}, and \ref{as:compressor} hold. Under Algorithm~\ref{Al:CP-SGD}, if $\alpha_x\in(0,\frac{1}{r})$ and $\{\omega_k\}$ is non-decreasing, we have
\begin{align*}
    &\E_{\xi_k}[V_{1,k+1}]\leq V_{1,k}-\left\|\mathbf{x}_k\right\|_{\frac{\eta_k \gamma_k}{2} \mathbf{L}-\frac{\eta_k}{2} \mathbf{K}-\frac{\eta_k}{2}(1+5 \eta_k) L_f^2 \boldsymbol{K}}^2\\
    &~~~+\left\|\hat{\mathbf{x}}_k\right\|_{\frac{3\eta_k^2 \gamma_k^2}{2} \mathbf{L}^2}^2+ \frac{\eta}{2}(\gamma+2 \omega)\bar{\lambda}_L\left\|\hat{\mathbf{x}}_k-\mathbf{x}_k\right\|^2\\
&~~~ -\eta_k \omega_k \hat{\mathbf{x}}_k^{\top} \mathbf{K}\left(\mathbf{v}_k+\frac{1}{\omega_k} \g^b_k\right)\\
&~~~+\frac{6 \eta_k^2 \omega_k^2 \bar{\lambda}_L+\eta_k \omega_k}{4}\left\|\mathbf{v}_k+\frac{1}{\omega_k} \g^b_k\right\|^2_\bp+2n\sigma^2\eta_k^2,\addtag\label{eq:lyapunovofv1}\\
&V_{2,k+1}\leq V_{2,k}+\left(1+b_k\right)\eta_k\omega_k(1+\beta_1)\hx_k^\top\bk\left(\vv_k+\frac{1}{\omega_k}\g_k^b\right) \\
&~~~+\|\hx_k\|_{\left(1+b_k\right) \left(\frac{\eta_k^2\omega_k}{2}\left(\omega_k+\gamma_k\right) \bl+\frac{\eta_k^2}{2}\bk\right)}\\
&~~~+\frac{1}{2}\left(b_k+b_k\beta_1+\frac{\eta_k}{2}+\frac{b_k\eta_k}{2}\right)\left\|\vv_k+\frac{1}{\omega_k} \g_k^b\right\|_\bp^2\\
&~~~+(1+b_k)\bigg(\left(\frac{(1+\beta_1)^2}{\eta_k\omega_k^2}+\frac{1+\beta_1}{2\omega_k^2}\right)\frac{1}{\underline{\lambda}_L}\\
&~~~+\frac{1}{2}\bigg)\eta_k^2L_f^2\E_{\xi_k}[\Vert\bg_k^s\Vert^2]\\
&~~~+(1+b_k)\eta_k\beta_1\hx_k^\top\bk(\g_{k+1}^b-\g_k^b)\\
&~~~+\frac{1}{2\underline{\lambda}_L}(b_k+b_k^2)(1+\beta_1)\Vert\g_{k+1}^b\Vert^2,\addtag\label{eq:lyapunovofv2}\\
   &V_{3,k+1}\leq V_{3,k}-(1+b_k)\eta_k\gamma_k\hx_k^\top\bk(\vv_k+\frac{1}{\omega_k}\g_k^b)\\
   &~~~+\Vert\hx_k\Vert^2_{\eta_k(\omega_k\bk+\frac{b_k\gamma_k}{8}\bl)+\eta_k^2(\omega_k^2\bk+\frac{b_k}{2}\bk-\omega_k\gamma_k\bl)}\\
&~~~+\Vert\x_k\Vert_{(\frac{\eta_k(\omega_k+2)}{4})\bk+(\frac{\eta_k}{4}+\frac{3\eta_k^2}{2})L_f^2\bk}^2\\
&~~~-(1+b_k)\frac{\eta_k\gamma_k}{\omega_k}\hx_k^\top\bk(\g_{k+1}^b-\g_k^b)+\frac{\eta_k}{8}\Vert\bg_k\Vert^2\\
 &~~~-\Vert\vv_k+\frac{1}{\omega_k}\g_k^b\Vert_{\eta_k(\omega_k-3\underline{\lambda}_L^{-1})\bp-\eta_k^2(\underline{\lambda}_L^{-1}-\frac{\omega_k^2}{2}\bar{\lambda}_L)\bp-2b_k\eta_k\gamma_k\bp}^2\\
 &~~~+\frac{b_k}{2}(\|\x_{k+1}\|^2_\bk+\|\g_{k+1}^b\|^2)\\
 &~~~+(\frac{1+\eta_k}{2\eta_k\omega_k^2\underline{\lambda}_L^2}+\frac{b_k\gamma_k^2}{2\omega_k^2}+\frac{1}{4})\eta_k^2L_f^2\E_{\xi_k}[\Vert\bg_k^s\Vert^2]+n\eta_k^2\sigma^2,\addtag\label{eq:lyapunovofv3}\\
 &V_{4,k+1}\leq V_{4,k}-\frac{\eta_k}{4}\|\bg_k\|^2+\frac{\eta_kL_f^2}{2}\|\x_k\|_\bk^2-\frac{\eta_k}{4}\|\bg_k^b\|^2\\
 &~~~+\frac{\eta_k^2L_f}{2}\E_{\xi_k}\Vert\bg_k^s\Vert^2,\addtag\label{eq:lyapunovofv4}\\
 &V_{5,k+1}\leq (1-\frac{\varphi_2}{2}-\frac{\varphi_2^2}{2}+4\eta_k^2\gamma_k^2\bar{\lambda}_L^2r_0(1+\frac{2}{\varphi_2}))\Vert\x_k-\x_{k}^c\Vert^2\\
 &~~~+\Vert\x_k\Vert^2_{4\eta_k^2(1+\frac{2}{\varphi_2})(\gamma_k^2\bar{\lambda}_L^2+2L_f^2)\bk}\\
	&~~~+\Vert\vv_k+\frac{1}{\omega_k}\g_k^b\Vert^2_{4\eta_k^2(1+\frac{2}{\varphi_2})\omega_k^2\underline{\lambda}_L\bp}+(1+\frac{2}{\varphi_2})8n\eta_k^2\sigma^2,\addtag\label{eq:lyapunovofv5}
\end{align*}
where
\begin{align*}
&V_{1,k+1}=\frac{1}{2}\Vert\x_{k+1}\Vert_\bk^2\\
&V_{2,k+1}=\frac{1}{2}\Vert\vv_{k+1}+\frac{1}{\omega_{k+1}}\g_{k+1}^b\Vert^2_{\bp+\beta_1\bp}\\
&V_{3,k+1}=\x_{k+1}^\top\bk\bp(\vv_{k+1}+\frac{1}{\omega_{k+1}}\g_{k+1}^b)\\
    &V_{4,k+1}=n(f(\bar{x}_{k+1})-f^*),\\
    &V_{5,k+1}=\Vert\x_{k+1}-\x^c_{k+1}\Vert^2,\\
&b_k=\frac{1}{\omega_k}-\frac{1}{\omega_{k+1}},\\
&\varphi_2=\alpha_xr\varphi.
\end{align*}

\end{lemma}

For simplicity of the proof, we first provide some useful properties and inequalities. The update equations~\eqref{eq:iterationxp2}, \eqref{eq:iterationv2}, \eqref{citerationx}, and~\eqref{citerationxc} can be rewritten as the following compact form
\begin{align*}
	&~\x_{k+1}=\x_{k}-\eta_k(\gamma_k \bl\hx_{k}+\omega_k  \vv_{k}+\mathbf{g}^s_k),\addtag\label{eq:iterationxp3}\\
	&~\vv_{k+1}=\vv_{k}+\eta_k\omega_k \bl\hx_{k},\addtag\label{eq:iterationv3}\\
	&~\hx_{k}=\x_{k}^c+C(\x_{k}-\x_{k}^c),\addtag\label{citerationx1}\\
	&~\x_{k+1}^c=(1-\alpha_x)\x_{k}^c+\alpha_x\hx_{k},\addtag\label{citerationxc1}
	\end{align*}
 where $\x_k\triangleq[x_{1,k}^\top,\dots,x_{n,k}\top]^\top\in\R^{nd} $, $\mathbf{g}^s_k\triangleq[\tilde{\nabla} f_{1,k})^\top,\dots,\tilde{\nabla} f_{n,k})^\top]\in\R^{nd}$, $\vv_k\triangleq[v_{1,k}^\top,\dots,v_{n,k}^\top]^\top\in\R^{nd}$, $\hx_k\triangleq\left[\hat{x}_{1,k}^\top,\dots,\hat{x}_{n,k}^\top\right]^\top\in\R^{nd}$, and $\x_k^c\triangleq[{x^c}_{1,k}^\top,\dots,{x^c}_{n,k}\top]^\top\in\R^{nd}$.
From~\eqref{eq:iterationv3}, the property of Laplacian matrix, and the fact that $\sum_{i=1}^n v_{i,0}=\mathbf{0}_d$, we have 
   \begin{align}\label{eq:propertyofv}
	\bv_{k+1}=\mathbf{0}.
   \end{align}
Then from~\eqref{eq:iterationxp3} and~\eqref{eq:propertyofv}, one obtains that
\begin{align}\label{eq:propertyofbx}
	\bx_{k+1}=\bx_k-\eta_k\bg_k^s.
\end{align}
From Assumption~\ref{as:smooth}, one obtains that
\begin{align}\label{eq:propertyofbg1}
	\Vert\g^b_k-\g_k\Vert^2\leq L_f^2\Vert\bx_k-\x_k\Vert^2\leq L_f^2\Vert\x_k\Vert^2_\bk.
\end{align}
Furthermore, we have following useful equations
\begin{align*}
	&\Vert\bg^b_k-\bg_k\Vert^2=\Vert\bh(\g^b_k-\g_k)\Vert^2\leq L_f^2\Vert\x_k\Vert^2_\bk,\addtag\label{eq:propertyofbg2}\\
	&\Vert\g^b_{k+1}-\g^b_k\Vert^2\leq L_f^2\Vert\bx_{k+1}-\bx_k\Vert^2\\
	&~~~~~~~~~~~~~~~~~\leq \eta_k^2L_f^2\Vert\bg_k^s\Vert^2\addtag\label{eq:propertyofbg3},
\end{align*}
where the first inequality comes from~\eqref{eq:propertyofbg1} and $\lb_\bh=1$; the last inequality comes from Assumption~\ref{as:smooth} and~\eqref{eq:iterationxp3}.
From Assumption~\ref{as:independent}--\ref{as:boundedvar}, we have
\begin{align*}
    &\E_{\xi_k}[\g^s_k]=\g_k,\addtag\label{eq:propertyofsg1}\\
    &\E_{\xi_k}[\Vert\g^s_k-\g_k\Vert^2]\leq n\sigma^2,\addtag\label{eq:propertyofsg2}\\
    &\E_{\xi_k}[\bg^s_k]=\E_{\xi_k}[\bh\g^s_k]=\bh\E_{\xi_k}[\g^s_k]=\bg_k.\addtag\label{eq:propertyofsg4}
\end{align*}
Combining~\eqref{eq:inequality2}~\eqref{eq:propertyofbg1}, and~\eqref{eq:propertyofsg2}, we have 
\begin{align*}
    \E_{\xi_k}[\Vert\g^s_k-\g_k^b\Vert^2]&\leq 2\E_{\xi_k}[\Vert\g^s_k-\g_k\Vert^2]+2\Vert\g^b_k-\g_k\Vert^2\\
    &\leq 2L_f^2\Vert\x_k\Vert^2_\bk+2n\sigma^2.\addtag\label{eq:propertyofsg3}
\end{align*}

\textbf{(i)} This part shows the upper bound of $V_{1,k+1}$
    \begin{align*}
\E_{\xi_k}[\frac{1}{2}\Vert&\x_{k+1}\Vert_\bk^2]=\E_{\xi_k}[\frac{1}{2}\Vert\x_{k}-\eta_k(\gamma_k \bl\hx_{k}+\omega_k  \vv_{k}+\g^s_k)\Vert_\bk^2]\\
& =\frac{1}{2}\left\|\mathbf{x}_k\right\|_{\mathbf{K}}^2-\eta_k \gamma_k \mathbf{x}_k^{\top} \mathbf{L} \hat{\mathbf{x}}_k+\left\|\hat{\mathbf{x}}_k\right\|_{\frac{\eta_k^2 \gamma_k^2}{2} \mathbf{L}^2}^2 \\
&~~~-\eta_k \omega_k\left(\mathbf{x}_k^{\top}-\eta_k \gamma_k \hat{\mathbf{x}}_k^{\top} \mathbf{L}\right) \mathbf{K}\left(\mathbf{v}_k+\frac{1}{\omega_k} \mathbf{g}_k\right)\\
&~~~+\E_{\xi_k}\left\|\mathbf{v}_k+\frac{1}{\omega_k} \mathbf{g}_k^s\right\|_{\frac{\eta_k^2 \omega_k^2}{2} \mathbf{K}}^2 \\
& =\frac{1}{2}\left\|\mathbf{x}_k\right\|_{\mathbf{K}}^2-\eta_k \gamma_k \mathbf{x}_k^{\top} \mathbf{L}\left(\mathbf{x}_k+\hat{\mathbf{x}}_k-\mathbf{x}_k\right)+\left\|\hat{\mathbf{x}}_k\right\|_{\frac{\eta_k^2 \gamma_k^2}{2}\mathbf{L}^2}^2 \\
&~~~-\eta_k \omega_k\left(\mathbf{x}_k^{\top}-\eta_k \gamma_k \hat{\mathbf{x}}_k^{\top} \mathbf{L}\right) \mathbf{K}\bigg(\mathbf{v}_k+\frac{1}{\omega_k} \g^b_k\\
&~~~+\frac{1}{\omega_k} \mathbf{g}_k-\frac{1}{\omega_k} \g^b_k\bigg) \\
& +\E_{\xi_k}\left\|\mathbf{v}_k+\frac{1}{\omega_k} \g^b_k+\frac{1}{\omega_k} \mathbf{g}_k^s-\frac{1}{\omega_k} \g^b_k\right\|_{\frac{\eta_k^2 \omega_k^2}{2} \mathbf{K}}^2 \\
& \leq \frac{1}{2}\left\|\mathbf{x}_k\right\|_{\mathbf{K}}^2-\left\|\mathbf{x}_k\right\|_{\eta_k\gamma_k\mathbf{L}}^2+\left\|\mathbf{x}_k\right\|_{\frac{\eta_k\gamma_k}{2} \mathbf{L}}^2\\
&~~~+\left\|\hat{\mathbf{x}}_k-\mathbf{x}_k\right\|_{\frac{\eta_k\gamma_k}{2} \mathbf{L}}^2 +\left\|\hat{\mathbf{x}}_k\right\|_{\frac{\eta_k^2 \gamma_k^2}{2} L^2}^2 \\
&~~~-\eta_k \omega_k \mathbf{x}_k^{\top} \mathbf{K}\left(\mathbf{v}_k+\frac{1}{\omega_k} \g^b_k\right)+\frac{\eta_k}{2}\left\|\mathbf{x}_k\right\|_{\mathbf{K}}^2 \\
&~~~+\frac{\eta_k}{2}\left\|\mathbf{g}_k-\g^b_k\right\|^2+\left\|\hat{\mathbf{x}}_k\right\|_{\frac{\eta_k^2 \gamma_k^2}{2} \mathbf{L}^2}^2\\
&~~~+\frac{\eta_k^2 \omega_k^2}{2}\left\|\mathbf{v}_k+\frac{1}{\omega_k} \g^b_k\right\|^2 +\left\|\hat{\mathbf{x}}_k\right\|_{\frac{\eta_k^2 \gamma_k^2}{2} \mathbf{L}^2}^2\\
& ~~~+\frac{\eta_k^2}{2}\left\|\mathbf{g}_k-\g^b_k\right\|^2  +\eta_k^2 \omega_k^2\left\|\mathbf{v}_k+\frac{1}{\omega_k} \g^b_k\right\|^2\\
&~~~+\eta_k^2\E_{\xi_k}\left\|\mathbf{g}_k^s-\g^b_k\right\|^2 \\
& =\frac{1}{2}\left\|\mathbf{x}_k\right\|_{\mathbf{K}}^2-\left\|\mathbf{x}_k\right\|_{\frac{\eta_k \gamma_k}{2} \mathbf{L}-\frac{\eta_k}{2} \mathbf{K}}^2+\left\|\hat{\mathbf{x}}_k\right\|_{\frac{3 \eta_k^2 \gamma_k^2}{2} \mathbf{L}^2}^2 \\
& +\frac{\eta_k}{2}(1+ \eta_k)\left\|\mathbf{g}_k-\g^b_k\right\|^2+\left\|\hat{\mathbf{x}}_k-\mathbf{x}_k\right\|_{\frac{\eta_k\gamma_k}{2}\bl}^2 \\
& -\eta_k \omega_k\left(\hat{\mathbf{x}}_k+\mathbf{x}_k-\hat{\mathbf{x}}_k\right)^{\top} \mathbf{K}\left(\mathbf{v}_k+\frac{1}{\omega_k} \g^b_k\right)\\
&~~~+\frac{3 \eta_k^2 \omega_k^2}{2}\left\|\mathbf{v}_k+\frac{1}{\omega_k} \g^b_k\right\|^2+\eta_k^2\E_{\xi_k}\left\|\mathbf{g}_k^s-\g^b_k\right\|^2 \\
& \leq \frac{1}{2}\left\|\mathbf{x}_k\right\|_{\mathbf{K}}^2-\left\|\mathbf{x}_k\right\|_{\frac{\eta_k \gamma_k}{2} \mathbf{L}-\frac{\eta_k}{2} \mathbf{K}}^2+\left\|\hat{\mathbf{x}}_k\right\|_{\frac{3\eta_k^2\gamma_k^2}{2} \mathbf{L}^2}^2 \\
&~~~+\frac{\eta_k}{2}(1+\eta_k)\left\|\mathbf{g}_k-\g^b_k\right\|^2\\
&~~~+\left\|\hx_k-\x_k\right\|_{\frac{\eta_k}{2}(\gamma_k \mathbf{L}+2 \omega_k \bar{\lambda}_L \mathbf{K})}^2 \\
&~~~-\eta_k \omega_k \hat{\mathbf{x}}_k^{\top} \mathbf{K}\left(\mathbf{v}_k+\frac{1}{\omega_k} \g^b_k\right)+\eta_k^2\E_{\xi_k}\left\|\mathbf{g}_k^s-\g^b_k\right\|^2\\
&~~~+\frac{6 \eta_k^2 \omega_k^2+\eta_k \omega_k \bar{\lambda}_L^{-1}}{4}\left\|\mathbf{v}_k+\frac{1}{\omega_k} \g^b_k\right\|^2,\\
&\leq  \frac{1}{2}\left\|\mathbf{x}_k\right\|_{\mathbf{K}}^2-\left\|\mathbf{x}_k\right\|_{\frac{\eta_k \gamma_k}{2} \mathbf{L}-\frac{\eta_k}{2} \mathbf{K}-\frac{\eta_k}{2}(1+5 \eta_k) L_f^2 \boldsymbol{K}}^2\\
&~~~+\left\|\hat{\mathbf{x}}_k\right\|_{\frac{3\eta_k^2 \gamma_k^2}{2} \mathbf{L}^2}^2 +\frac{\eta}{2}(\gamma+2 \omega)\bar{\lambda}_L\left\|\hat{\mathbf{x}}_k-\mathbf{x}_k\right\|^2 \\
&~~~ -\eta_k \omega_k \hat{\mathbf{x}}_k^{\top} \mathbf{K}\left(\mathbf{v}_k+\frac{1}{\omega_k} \g^b_k\right)\\
&~~~+\frac{6 \eta_k^2 \omega_k^2 \bar{\lambda}_L+\eta_k \omega_k}{4}\left\|\mathbf{v}_k+\frac{1}{\omega_k} \g^b_k\right\|^2_\bp+2n\sigma^2\eta_k^2,\addtag\label{eq:upperboundofx1}
\end{align*}
where the first and second equalities comes from~\eqref{eq:iterationxp3},~\eqref{eq:propertyofp1} and~\eqref{eq:propertyofsg1}; the first inequality come from~\eqref{eq:inequality1} and~\eqref{eq:propertyofk}; the second inequality comes from~\eqref{eq:inequality1} and $\lb_\bk=1$,; the last inequality comes from~\eqref{eq:propertyofk1},~\eqref{eq:propertyofp1},~\eqref{eq:propertyofbg1}, and~\eqref{eq:propertyofsg3}. 

\textbf{(ii)} This part shows the upper bound of $V_{2,k+1}$
From the sequence $\{\omega_k\}$ in non-decreasing and~\eqref{eq:inequality1}, one obtains that
\begin{align*}
    V_{2,k+1}&=\frac{1}{2}\Vert\vv_{k+1}+\frac{1}{\omega_{k+1}}\g_{k+1}^b\Vert^2_{\bp+\beta_1\bp}\\
    &=\frac{1}{2}\Vert\vv_{k+1}+\frac{1}{\omega_{k}}\g_{k+1}^b+(\frac{1}{\omega_{k+1}}-\frac{1}{\omega_{k}})\g_{k+1}^b\Vert^2_{\bp+\beta_1\bp}\\
    &\leq\frac{1}{2}(1+b_k)\Vert\vv_{k+1}+\frac{1}{\omega_{k}}\g_{k+1}^b\Vert^2_{\bp+\beta_1\bp}\\
    &~~~+\frac{1}{2}(b_k+b_k^2)\Vert\g_{k+1}^b\Vert^2_{\bp+\beta_1\bp}.\addtag\label{eq:propertyofv21}
\end{align*}
With respect to $\Vert\vv_{k+1}+\frac{1}{\omega_{k}}\g_{k+1}^b\Vert^2_{\bp+\beta_1\bp}$, we have
\begin{align*}
	\frac{1}{2}\Vert&\vv_{k+1}+\frac{1}{\omega_k}\g_{k+1}^b\Vert^2_{\bp+\beta_1\bp}\\
	&=\frac{1}{2}\Vert\vv_k\!+\!\frac{1}{\omega_k}\g_k^b+\!\eta_k\omega_k\bl\hx_k\!+\!\frac{1}{\omega_k}(\g_{k+1}^b-\g_k^b)\Vert^2_{\bp+\beta_1\bp}\\
	&= \frac{1}{2}\Vert\vv_k+\frac{1}{\omega_k}\g_k^b\Vert^2_{\bp+\beta_1\bp}\\
 &~~~+\eta_k\omega_k(1+\beta_1)\hx_k^\top\bk\left(\vv_k+\frac{1}{\omega_k}\g_k^b\right)\\
  &~~~+\Vert\hx_k\Vert_{\frac{\eta_k^2\omega_k^2}{2}(1+\beta_1)\bl}^2+\frac{1}{2\omega_k^2}\Vert \g_{k+1}^b-\g_k^b\Vert^2_{\bp+\beta_1\bp}\\
	&~~~+\frac{1}{\omega_k}(\g_{k+1}^b-\g_k^b)^\top(\bp+\beta_1\bp)\left(\vv_k+\frac{1}{\omega_k}\g_k^b\right)\\
	&~~~+\eta_k\hx_k^\top(\bk+\beta_1\bk)(\g_{k+1}^b-\g_k^b)\\
	&\leq V_{2,k}+\eta_k\omega_k(1+\beta_1)\hx_k^\top\bk\left(\vv_k+\frac{1}{\omega_k}\g_k^b\right)\\
 &~~~+\Vert\hx_k\Vert_{\frac{\eta_k^2\omega_k^2}{2}(1+\beta_1)\bl}^2+\frac{1}{2\omega_k^2}\Vert \g_{k+1}^b-\g_k^b\Vert^2_{\bp+\beta_1\bp}\\
	&~~~+\Vert\vv_k+\frac{1}{\omega_k}\g_k^b\Vert^2_{\frac{\eta_k}{4}\bp}+\Vert\g_{k+1}^b-\g_k^b\Vert^2_{\frac{(1+\beta_1)^2}{\eta_k\omega_k^2}\bp}\\
	&~~~+\Vert\hx_k\Vert_{\frac{\eta_k^2}{2}\bk}^2\!+\!\frac{1}{2}\Vert\g_{k+1}^b-\g_k^b\Vert^2+\eta_k\beta_1\hx_k^\top\bk(\g_{k+1}^b-\g_k^b)\\	&=V_{2,k}+\eta_k\omega_k(1+\beta_1)\hx_k^\top\bk\left(\vv_k+\frac{1}{\omega_k}\g_k^b\right)\\	&~~~+\Vert\hx_k\Vert^2_{\frac{\eta_k^2\omega_k^2}{2}(1+\beta_1)\bl+\frac{\eta_k^2}{2}\bk}+\Vert\vv_k+\frac{1}{\omega_k}\g_k^b\Vert^2_{\frac{\eta_k}{4}\bp}\\
	&~~~+\Vert\g_{k+1}^b-\g_k^b\Vert^2_{(\frac{(1+\beta_1)^2}{\eta_k\omega_k^2}+\frac{1+\beta_1}{2\omega_k^2})\bp}+\frac{1}{2}\Vert\g_{k+1}^b-\g_k^b\Vert^2\\	&~~~+\eta_k\beta_1\hx_k^\top\bk(\g_{k+1}^b-\g_k^b)\\
 &\leq V_{2,k}+\eta_k\omega_k(1+\beta_1)\hx_k^\top\bk\left(\vv_k+\frac{1}{\omega_k}\g_k^b\right)\\	&~~~+\Vert\hx_k\Vert^2_{\frac{\eta_k^2\omega_k^2}{2}(1+\beta_1)\bl+\frac{\eta_k^2}{2}\bk}+\Vert\vv_k+\frac{1}{\omega_k}\g_k^b\Vert^2_{\frac{\eta_k}{4}\bp}\\
	&~~~+\left(\left(\frac{(1+\beta_1)^2}{\eta_k\omega_k^2}+\frac{1+\beta_1}{2\omega_k^2}\right)\frac{1}{\underline{\lambda}_L}+\frac{1}{2}\right)\eta_k^2L_f^2\E_{\xi_k}[\Vert\bg_k^s\Vert^2]\\	&~~~+\eta_k\beta_1\hx_k^\top\bk(\g_{k+1}^b-\g_k^b),\addtag\label{eq:propertyofv22}
\end{align*}
where $\tilde{b}_1=(\frac{1}{\omega_k^2}+\frac{1}{2\eta_k\omega_k})(1+\frac{\gamma_k}{\omega_k})\frac{1}{\underline{\lambda}_L}+\frac{1}{2}$; the first equality comes from~\eqref{eq:iterationv3}; the second equality comes from~\eqref{eq:propertyofk} and~\eqref{eq:propertyofp}; the first inequality comes from~\eqref{eq:inequality1}; the last inequality comes from~\eqref{eq:propertyofp1} and~\eqref{eq:propertyofbg3}. Combining~\eqref{eq:propertyofp1},~\eqref{eq:propertyofv21} and~\eqref{eq:propertyofv22}, we have
\begin{align*}
        V_{2,k+1}&\leq V_{2,k}+\left(1+b_k\right)\eta_k\omega_k(1+\beta_1)\hx_k^\top\bk\left(\vv_k+\frac{1}{\omega_k}\g_k^b\right) \\
&+\|\hx_k\|_{\left(1+b_k\right) \left(\frac{\eta_k^2\omega_k}{2}\left(\omega_k+\gamma_k\right) \bl+\frac{\eta_k^2}{2}\bk\right)}\\
&~~~+\frac{1}{2}\left(b_k+b_k\beta_1+\frac{\eta_k}{2}+\frac{b_k\eta_k}{2}\right)\left\|\vv_k+\frac{1}{\omega_k} \g_k^b\right\|_\bp^2\\
&~~~+(1+b_k)\left(\left(\frac{(1+\beta_1)^2}{\eta_k\omega_k^2}+\frac{1+\beta_1}{2\omega_k^2}\right)\frac{1}{\underline{\lambda}_L}+\frac{1}{2}\right)\\
&~~~\eta_k^2L_f^2\E_{\xi_k}[\Vert\bg_k^s\Vert^2]\\
&~~~+(1+b_k)\eta_k\beta_1\hx_k^\top\bk(\g_{k+1}^b-\g_k^b)\\
&~~~+\frac{1}{2\underline{\lambda}_L}(b_k+b_k^2)(1+\beta_1)\Vert\g_{k+1}^b\Vert^2.
\end{align*}

\textbf{(iii)} This part shows the upper bound of $V_{3,k+1}$. Similar to~\eqref{eq:propertyofv21}, one obtains that
\begin{align*}
V_{3,k+1}&=\x_{k+1}^\top\bk\bp(\vv_{k+1}+\frac{1}{\omega_{k+1}}\g_{k+1}^b)\\
&=\x_{k+1}^\top\bk\bp(\vv_{k+1}+\frac{1}{\omega_{k}}\g_{k+1}^b+(\frac{1}{\omega_{k+1}}-\frac{1}{\omega_{k}})\g_{k+1}^b)\\
&\leq \x_{k+1}^\top\bk\bp(\vv_{k+1}+\frac{1}{\omega_{k}}\g_{k+1}^b)+\frac{b_k}{2}(\|\x_{k+1}\|^2_\bk+\|\g_{k+1}^b\|^2).\addtag\label{eq:propertyofv31}
\end{align*}
Regards to the first term of~\eqref{eq:propertyofv31}, it  holds that
\begin{align*}	
\E_{\xi_k}[\x_{k+1}^\top&\bk\bp(\vv_{k+1}+\frac{1}{\omega_k}\g_{k+1}^b)]\\
	&=\E_{\xi_k}[(\x_k-\eta_k(\gamma_k\bl\hx_k+\omega_k\vv_k+\g_k^b+\g_k^s-\g_k^b))^\top\\
	&~~~\bk\bp(\vv_k +\frac{1}{\omega_k}\g_k^b+\eta_k\omega_k\bl\hx_k+\frac{1}{\omega_k}(\g_{k+1}^b-\g_k^b))]\\
	&=(\x_k^\top\bk\bp-\eta_k(\gamma_k+\eta_k\omega_k^2)\hx_k^\top\bk)(\vv_k+\frac{1}{\omega_k}\g_k^b)\\
	&~~~+\eta_k\omega_k\x_k^\top\bk\hx_k-\Vert\hx_k\Vert^2_{\eta_k^2\gamma_k\omega_k\bl}\\
 &~~~+\frac{1}{\omega_k}(\x_k^\top\bk\bp-\eta_k\gamma_k\hx_k^\top\bk)(\g_{k+1}^b-\g_k^b)\\
 &~~~-\eta_k(\omega_k\vv_k\!+\!\g_k^b+\!\g_k\!-\!\g_k^b\!-\bg_k)^\top\bp(\vv_k+\frac{1}{\omega_k}\g_k^b)\\
 &~~~-\eta_k(\vv_k+\frac{1}{\omega_k}\g_k^b)^\top\bp\bk(\g_{k+1}^b-\g_k^b)\\
	&~~~-\E_{\xi_k}[\eta_k(\g_k^s-\g_k^b)^\top(\eta_k\omega_k\bk\hx_k\\
 &~~~+\frac{1}{\omega_k}\bk\bp(\g_{k+1}^b-\g_k^b))]\\
	&\leq(\x_k^\top\bk\bp-\eta_k\gamma_k\hx_k^\top\bk)(\vv_k+\frac{1}{\omega_k}\g_k^b)+\Vert\hx_k\Vert_{\frac{\eta_k^2\omega_k^2}{2}\bk}^2\\
	&~~~+\Vert\vv_k+\frac{1}{\omega_k}\g_k^b\Vert_{\frac{\eta_k^2\omega_k^2}{2}}^2+\Vert\x_k\Vert_{\frac{\eta_k\omega_k}{4}\bk}^2\\
 &~~~+\Vert\hx_k\Vert_{\eta_k\omega_k(\bk-\eta_k\gamma_k\bl)}^2+\Vert\x_k\Vert_{\frac{\eta_k}{2}\bk}^2\\
	&~~~+\E_{\xi_k}[\Vert\g_{k+1}^b-\g_k^b\Vert^2_{\frac{1}{2\eta_k\omega_k^2}\bp^2}]\\
	&~~~-\frac{\eta_k\gamma_k}{\omega_k}\hx_k^\top\bk(\g_{k+1}^b-\g_k^b)-\Vert\vv_k+\frac{1}{\omega_k}\g_k^b\Vert_{\eta_k\omega_k\bp}^2\\
	&~~~+\frac{\eta_k}{4}\Vert\g_k-\g_k^b\Vert^2+\frac{\eta_k}{8}\Vert\bg_k\Vert^2\\
	&~~~+\Vert\vv_k+\frac{1}{\omega_k}\g_k^b\Vert_{3\eta_k\bp^2}^2+\Vert\vv_k+\frac{1}{\omega_k}\g_k^b\Vert_{\eta_k^2\bp^2}^2\\
	&~~~+\frac{1}{4}\Vert\g_{k+1}^b-\g_k^b\Vert^2+\frac{\eta_k^2}{2}\Vert\g_{k}-\g_k^b\Vert^2+\Vert\hx_k\Vert_{\frac{\eta_k^2\omega_k^2}{2}\bk}^2\\
	&~~~+\E_{\xi_k}[\frac{\eta_k^2}{2}\Vert\g_{k}^s-\g_k^b\Vert^2]+\E_{\xi_k}[\Vert\g_{k+1}^b-\g_k^b\Vert_{\frac{1}{2\omega_k^2}\bp^2}^2]\\
 &=(\x_k^\top\bk\bp-\eta_k\gamma_k\hx_k^\top\bk)(\vv_k+\frac{1}{\omega_k}\g_k^b)\\	&~~~+\Vert\x_k\Vert_{\frac{\eta_k(\omega_k+2)}{4}\bk}^2+\Vert\hx\Vert^2_{\eta_k\omega_k\bk+\eta_k^2(\omega_k^2\bk-\omega_k\gamma_k\bl)}\\
	&~~~+(\frac{\eta_k}{4}+\frac{\eta_k^2}{2})\Vert\g_{k}-\g_k^b\Vert^2\\
 &~~~-\frac{\eta_k\gamma_k}{\omega_k}\hx_k^\top\bk(\g_{k+1}^b-\g_k^b)+\frac{\eta_k}{8}\Vert\bg_k\Vert^2\\
 &~~~-\Vert\vv_k+\frac{1}{\omega_k}\g_k^b\Vert_{\eta_k\omega_k\bp-3\eta_k\bp^2-\eta_k^2\bp^2-\frac{\eta_k^2\omega_k^2}{2}\bi_{nd}}^2\\
 &~~~+\E_{\xi_k}[\frac{\eta_k^2}{2}\Vert\g_{k}^s-\g_k^b\Vert^2]\\
	&~~~+\E_{\xi_k}[\Vert\g_{k+1}^b-\g_k^b\Vert^2_{\frac{1+\eta_k}{2\eta_k\omega_k^2}\bp^2+\frac{1}{4}\bi_{nd}}]\\
	&\leq \x_k^\top\bk\bp(\vv_k+\frac{1}{\omega_k}\g_k^b)\\
	&~~~-(1+b_k)\eta_k\gamma_k\hx_k^\top\bk(\vv_k+\frac{1}{\omega_k}\g_k^b)\\ &~~~+\Vert\hx\Vert^2_{\eta_k\omega_k\bk+\eta_k^2(\omega_k^2\bk-\omega_k\gamma_k\bl)}\\	&~~~+\Vert\x_k\Vert_{\frac{\eta_k(\omega_k+2)}{4}\bk+(\frac{\eta_k}{4}+\frac{\eta_k^2}{2})L_f^2\bk}^2\\
  &~~~-(1+b_k)\frac{\eta_k\gamma_k}{\omega_k}\hx_k^\top\bk(\g_{k+1}^b-\g_k^b)+\frac{\eta_k}{8}\Vert\bg_k\Vert^2\\
 &~~~-\Vert\vv_k+\frac{1}{\omega_k}\g_k^b\Vert_{\eta_k(\omega_k-3\underline{\lambda}_L^{-1})\bp-\eta_k^2(\underline{\lambda}_L^{-1}-\frac{\omega_k^2}{2}\bar{\lambda}_L)\bp}^2\\
 &~~~+\E_{\xi_k}[\frac{\eta_k^2}{2}\Vert\g_{k}^s-\g_k^b\Vert^2]\\
	&~~~+\E_{\xi_k}[\Vert\g_{k+1}^b-\g_k^b\Vert^2_{\frac{1+\eta_k}{2\eta_k\omega_k^2\underline{\lambda}_\bl^2}\bi_{nd}+\frac{1}{4}\bi_{nd}}]\\
 &~~~+b_k\eta_k\gamma_k\hx_k^\top\bk(\vv_k+\frac{1}{\omega_k}\g_k^b)\\
	&~~~+b_k\frac{\eta_k\gamma_k}{\omega_k}\hx_k^\top\bk(\g_{k+1}^b-\g_k^b)\\
 &\leq \x_k^\top\bk\bp(\vv_k+\frac{1}{\omega_k}\g_k^b)\\
	&~~~-(1+b_k)\eta_k\gamma_k\hx_k^\top\bk(\vv_k+\frac{1}{\omega_k}\g_k^b)\\ &~~~+\Vert\hx\Vert^2_{\eta_k(\omega_k\bk+\frac{b_k\gamma_k}{8}\bl)+\eta_k^2(\omega_k^2\bk+\frac{b_k}{2}\bk-\omega_k\gamma_k\bl)}\\
 &~~~+\Vert\x_k\Vert_{(\frac{\eta_k(\omega_k+2)}{4})\bk+(\frac{\eta_k}{4}+\frac{3\eta_k^2}{2})L_f^2\bk}^2\\
  &~~~-(1+b_k)\frac{\eta_k\gamma_k}{\omega_k}\hx_k^\top\bk(\g_{k+1}^b-\g_k^b)+\frac{\eta_k}{8}\Vert\bg_k\Vert^2\\
 &~~~-\Vert\vv_k+\frac{1}{\omega_k}\g_k^b\Vert_{\eta_k(\omega_k-3\underline{\lambda}_L^{-1})\bp-\eta_k^2(\underline{\lambda}_L^{-1}-\frac{\omega_k^2}{2}\bar{\lambda}_L)\bp-2b_k\eta_k\gamma_k\bp}^2\\
 &~~~+(\frac{1+\eta_k}{2\eta_k\omega_k^2\underline{\lambda}_L^2}+\frac{b_k\gamma_k^2}{2\omega_k^2}+\frac{1}{4})\eta_k^2L_f^2\E_{\xi_k}[\Vert\bg_k^s\Vert^2]+n\eta_k^2\sigma^2,\addtag\label{eq:propertyofv32}
\end{align*}
where the first equality comes from~\eqref{eq:iterationxp3} and~\eqref{eq:iterationv3}; the second equality holds due to~\eqref{eq:propertyofk1},~\eqref{eq:propertyofp},~\eqref{eq:propertyofsg1},~\eqref{eq:propertyofsg2}, and the fact that $\bk=\bi-\bh$; the first inequality comes from~\eqref{eq:inequality1}; the second inequality holds due to~\eqref{eq:propertyofp1} and~\eqref{eq:propertyofbg1}; the last inequality holds due to~\eqref{eq:inequality1},~\eqref{eq:propertyofbg3}, and~\eqref{eq:propertyofsg3}. 

\textbf{(iv)} This part shows the upper bound of $V_{4,k+1}$.
\begin{align*}
	V_{4,k+1}&=n(f(\bar{x}_{k+1})-f^*)\\
 &=\tilde{f}(\bx_k)-nf^*+\tilde{f}(\bx_{k+1})-\tilde{f}(\bx_k)\\
	&\leq \tilde{f}(\bx_k)-nf^*-\eta_k(\bg_k^s)^\top\bg_k^b+\frac{\eta_k^2L_f}{2}\E_{\xi_k}\Vert\bg_k^s\Vert^2\\
	&=\tilde{f}(\bx_k)-nf^*-\frac{\eta_k}{2}\bg_k^\top(\bg_k^b+\bg_k-\bg_k)\\
	&~~~-\frac{\eta_k}{2}(\bg_k-\bg_k^b+\bg_k^b)^\top(\bg_k^b)+\frac{\eta_k^2L_f}{2}\E_{\xi_k}\Vert\bg_k^s\Vert^2\\
 &\leq\tilde{f}(\bx_k)-nf^*-\frac{\eta_k}{4}\|\bg_k\|^2+\frac{\eta_k}{2}\|\bg_k^b-\bg_k\|^2\\
	&~~~-\frac{\eta_k}{4}\|\bg_k^b\|^2+\frac{\eta_k^2L_f}{2}\E_{\xi_k}\Vert\bg_k^s\Vert^2\\
 &\leq\tilde{f}(\bx_k)-nf^*-\frac{\eta_k}{4}\|\bg_k\|^2+\frac{\eta_kL_f^2}{2}\|\x_k\|_\bk^2\\
	&~~~-\frac{\eta_k}{4}\|\bg_k^b\|^2+\frac{\eta_k^2L_f}{2}\E_{\xi_k}\Vert\bg_k^s\Vert^2,\addtag\label{eq:upperboundofax1}
\end{align*}
where the third equality holds due to~\eqref{eq:propertyofsg4}; the first inequality comes from~\eqref{eq:propertyofbx}, Assumption~\ref{as:smooth}, and the fact that $\bh=\bh\bh$; the second inequality holds due to~\eqref{eq:inequality1}; the last inequality comes from~\eqref{eq:inequality1},~\eqref{eq:inequality2}, and~\eqref{eq:propertyofbg2}.

\textbf{(v)} This part shows the upper bound of $V_{5,k+1}$
\begin{align*}
    V_{5,k+1}&=\|\x_{k+1}-\x_k^c\|^2\\
&=\Vert\x_{k+1}-\x_k+\x_k-\x_{k}^c-\alpha_xr\frac{C(\x_k-\x_{k}^c)}{r}\Vert^2\\
&\leq (1+s)(\alpha_xr(1-\varphi)+(1-\alpha_xr))\Vert\x_k-\x_{k}^c\Vert^2\\
&~~~+(1+\frac{1}{s})\Vert\x_{k+1}-\x_k\Vert^2\\
&\leq (1-\varphi_2-\frac{\varphi_2^2}{2})\Vert\x_k-\x_{k}^c\Vert^2\\
&~~~+(1+\frac{2}{\varphi_2})\Vert\x_{k+1}-\x_k\Vert^2,\addtag\label{eq:upperboundofc1}
\end{align*}
where the first equality comes from~\eqref{citerationx1} and~\eqref{citerationxc1}; the first inequality comes from~\eqref{eq:inequality1}; the second inequality follows by denoting $\varphi_2=\alpha_xr\varphi$, choosing $s=\frac{\varphi_2}{2}$, and $\alpha_xr<1$. We have
\begin{align*}
	\Vert\x_{k+1}&-\x_k\Vert^2\\
 &=\|\eta_k\gamma_k\bl\hx_k+\omega_k\vv_k+\g_k^s\|^2\\
	&=\eta_k^2\Vert(\gamma_k\bl(\hx_k-\x_k)+\gamma_k\bl\x_k+\omega_k\vv_k+\g_k^b\\
	&~~~+\g_k^s-\g_k^b)\Vert^2\\
	&\leq 4\eta_k^2(\Vert\gamma_k\bl(\hx_k-\x_k)\Vert^2+\Vert\omega_k\vv_k+\g_k^b\Vert^2\\
	&~~~+\Vert\gamma_k\bl\x_k\Vert^2+\Vert\g_k^s-\g_k^b\Vert^2)\\
	&\leq 4\eta_k^2(\gamma_k^2\bar{\lambda}_L^2r_0\Vert\x^c_k-\x_k\Vert^2\!+\!\Vert\vv_k\!+\!\frac{1}{\omega_k}\g_k^b\Vert^2_{\omega_k^2\underline{\lambda}_L\bp}\\
&~~~+\Vert\x_k\Vert^2_{(\gamma_k^2\bar{\lambda}_L^2+2L_f^2)\bk}+2n\sigma^2),\addtag\label{eq:upperboundofc2}
\end{align*}
where the first equality holds due to~\eqref{eq:iterationxp3}; the first inequality holds due to Jensen's inequality; the last inequality holds due to~\eqref{eq:propertyofcompressors1},~\eqref{eq:propertyofp1}, and~\eqref{eq:propertyofsg3}. Combining~\eqref{eq:upperboundofc1}--\eqref{eq:upperboundofc2}, one obtains that
\begin{align*}
	\Vert\x_{k+1}-\x^c_{k+1}\Vert^2
	&\leq (1-\frac{\varphi_2}{2}-\frac{\varphi_2^2}{2}\\
	&~~~+4\eta_k^2\gamma_k^2\bar{\lambda}_L^2r_0(1+\frac{2}{\varphi_2}))\Vert\x_k-\x_{k}^c\Vert^2\\
 &~~~+\Vert\x_k\Vert^2_{4\eta_k^2(1+\frac{2}{\varphi_2})(\gamma_k^2\bar{\lambda}_L^2+2L_f^2)\bk}\\
	&~~~+\Vert\vv_k+\frac{1}{\omega_k}\g_k^b\Vert^2_{4\eta_k^2(1+\frac{2}{\varphi_2})\omega_k^2\underline{\lambda}_L\bp}\\&~~~+(1+\frac{2}{\varphi_2})8n\eta_k^2\sigma^2,\addtag\label{eq:upperboundofc3}
\end{align*}

\subsection{The proof of Theorem~\ref{theo:convergence1}}
For simplicity of the proof, we also denote some notations and a useful auxiliary function
\begin{align*}
&\epsilon_1=\frac{\gamma}{2}\bl-(\frac{\omega+4}{4}+\frac{5}{4}L_f^2)\bk\\
&\epsilon_2=(12+\frac{16}{\varphi_2})L_f^2+(4+\frac{8}{\varphi_2}\gamma^2\bar{\lambda}_L^2)\\
&\epsilon_3=\frac{3\gamma^2}{2}\bl^2+(\frac{\omega^2-\omega\gamma}{2})\bl+(\frac{1}{2}+\omega^2)\bk\\
&\epsilon_4=\frac{3\omega-1}{4}-3\underline{\lambda}_L^{-1}\\
&\epsilon_5=\omega^2\bar{\lambda}_L+\underline{\lambda}_L^{-1}+(4+\frac{8}{\varphi_2})\omega^2\underline{\lambda}_L\\
&\epsilon_6=\frac{1}{8}-(\frac{2(1+\beta_1)}{\omega^2\underline{\lambda}_L}+\frac{1}{\omega^2\underline{\lambda}_L^2})L_f^2\\
&\epsilon_7=(\frac{1+\beta_1}{\omega^2\underline{\lambda}_L}+\frac{1}{\omega^2\underline{\lambda}_L^2}+\frac{3}{2})L_f^2+L_f\\
&\epsilon_8=(\frac{2(1+\beta_1)^2}{\eta\omega^2\underline{\lambda}_L}+\frac{1+\beta_1}{\omega^2\underline{\lambda}_L}+\frac{1+\eta}{\eta\omega^2\underline{\lambda}_L^2}+\frac{3}{2})L_f^2+L_f\\
&\epsilon_9=11+\frac{16}{\varphi_2}\\
&\epsilon_{10}=\frac{\varphi_2}{2}+\frac{\varphi_2^2}{2}\\
&\epsilon_{11}=\frac{1}{2}(\gamma+2 \omega)\bar{\lambda}_Lr_0+2\omega r_0\\
&\epsilon_{12}=\frac{(8+7\varphi_2)\gamma^2\bar{\lambda}_L^2r_0}{\varphi_2}+(1+2\omega^2)r_0\\
&\bar{\epsilon}_1=\frac{\gamma}{2}\underline{\lambda}_L-(\frac{\omega+4}{4}+\frac{5}{4}L_f^2)\\
&\bar{\epsilon}_3=\frac{1}{2}+\omega^2+\frac{3\gamma^2\bar{\lambda}_L^2}{2}\\
&\tilde{\epsilon}_1=\frac{\gamma}{2}\underline{\lambda}_L-(\frac{9\omega+4}{4}+\frac{5}{4}L_f^2)\\
&\tilde{\epsilon}_2=(12+\frac{16}{\varphi_2})L_f^2+(4+\frac{8}{\varphi_2}\gamma^2\bar{\lambda}_L^2)+1+2\omega^2+3\gamma^2\bar{\lambda}_L^2\\
&\beta_3=\max\{\frac{4+5L_f^2}{\beta_5},\frac{12\underline{\lambda}_L+1}{3},\sqrt{\beta_6},\frac{\beta_2}{\beta_4},4\beta_2 L_f\}\\
&\beta_4=\min\{\frac{\tilde{\epsilon}_1}{\tilde{\epsilon}_2},\frac{\epsilon_4}{\epsilon_5},\frac{\epsilon_6}{\epsilon_7},\frac{\sqrt{\epsilon_{11}^2+4\epsilon_{10}\epsilon_{12}}-\epsilon_{11}}{2\epsilon_{12}},1\}\\
&\beta_5>0\\
&\beta_6=(\frac{16(1+\beta_1)}{\underline{\lambda}_L}+\frac{8}{\underline{\lambda}_L^2})L_f^2\\
&\check{c}_1=\frac{\gamma \underline{\lambda}_L-\omega}{2 \gamma \underline{\lambda}_L}\\
&c_0=\max\{\frac{9+\beta_5}{2\underline{\lambda}_L},1\}\\
&c_1=(\frac{2(1+\beta_1)^2}{\beta_2\beta_3\underline{\lambda}_L}+\frac{1+\beta_1}{\beta_3^2\underline{\lambda}_L}+\frac{1}{\beta_2\beta_3\underline{\lambda}_L^2}++\frac{1}{\beta_3^2\underline{\lambda}_L^2}+\frac{3}{2})L_f^2+L_f\\
&c_2=\eta\tilde{\epsilon}_1-\eta^2\tilde{\epsilon}_2.
\end{align*}
\begin{lemma}
    Suppose Assumptions~\ref{as:strongconnected}--\ref{as:boundedvar} and~\ref{as:compressor} hold. If $\gamma_k=\gamma=\beta_1\omega$, $\beta_1>1$, $\omega_k=\omega$, $\alpha_x\in(0,\frac{1}{r})$, and $\eta_k=\eta$, it  holds that
    \begin{align*}
        \E_{\xi_k}[V_{k+1}]&\leq V_k-\|\x_k\|_{(\eta\tilde{\epsilon}_1-\eta^2\tilde{\epsilon}_2)\bk}-\Vert\vv_k+\frac{1}{\omega}\g_k^b\Vert^2_{\eta(\epsilon_4-\eta\epsilon_5)\bp}\\
 &~~~-\eta(\epsilon_6-\eta\epsilon_7)\|\bg_k\|-\frac{\eta}{4}\|\bg_k^b\|^2+\epsilon_8\sigma^2\eta^2+\epsilon_9n\sigma^2\eta^2\\
 &~~~-(\epsilon_{10}-\eta\epsilon_{11}-\eta^2\epsilon_{12})\Vert\x_k-\x_{k}^c\Vert^2,\addtag\label{eq:upperofV14}
    \end{align*}
    where $V_{k+1}=\sum_{i=1}^5 V_{i,k+1}$.
\end{lemma}
\begin{proof}
   
We first consider the term $\E_{\xi_k}[\Vert\bg_k^s\Vert^2$
\begin{align*}
    \E_{\xi_k}[\Vert\bg_k^s\Vert^2]&= \E_{\xi_k}[\Vert\bg_k^s-\bg_k+\bg_k\Vert^2]\\
    &\leq 2\E_{\xi_k}[\Vert\bg_k^s-\bg_k\Vert^2]+2\|\bg_k\|^2\\
    &=\frac{2}{n}\E_{\xi_k}[\Vert \sum_{i=1}^n g_{i,k}^s-g_{i,k}\Vert^2]+2\|\bg_k\|^2\\
    &=\frac{2}{n}\sum_{i=1}^n \E_{\xi_k}[\Vert g_{i,k}^s-g_{i,k}\Vert^2]+2\|\bg_k\|^2\\
    &\leq 2\sigma^2+2\|\bg_k\|^2,\addtag\label{eq:upperofV11}
\end{align*}
where the first inequality holds due to~\eqref{eq:inequality2}; the last equality holds due to Assumptions~\ref{as:independent} and~\ref{as:unbiase}; the last inequality holds due to~\eqref{eq:propertyofsg2}. We then consider the term $\|\hx_k\|_\bk^2$
\begin{align*}
   \|\hx_k\|_\bk^2=\|\hx_k-\x_k+\x_k\|_\bk^2\leq 2\|\hx_k-\x_k\|^2+2\|\x_k\|_\bk^2.\addtag\label{eq:upperofV12}
\end{align*}

 Since $\gamma_k=\gamma=\beta_1\omega$, $\omega_k=\omega$, $\eta_k=\eta$, and~\eqref{eq:upperofV11}, from Lemma~\ref{lemma:lyapunov}, we have
    \begin{align*}
\E_{\xi_k}&[V_{k+1}]\\
&\leq V_k-\left\|\mathbf{x}_k\right\|_{\frac{\eta \gamma}{2} \mathbf{L}-\frac{\eta}{2} \mathbf{K}-\frac{\eta}{2}(1+5 \eta) L_f^2 \boldsymbol{K}}^2+\left\|\hat{\mathbf{x}}_k\right\|_{\frac{3\eta^2 \gamma^2}{2} \mathbf{L}^2}^2 \\
	&~~~+\frac{\eta}{2}(\gamma+2 \omega)\bar{\lambda}_Lr_0\left\|\hat{\mathbf{x}}_k-\mathbf{x}^c_k\right\|^2 \\
&~~~ +\frac{6 \eta^2 \omega^2 \bar{\lambda}_L+\eta \omega+\eta}{4}\left\|\mathbf{v}_k+\frac{1}{\omega} \g^b_k\right\|^2_\bp+2n\sigma^2\eta^2\\
	&~~~+\|\hx_k\|_{ \frac{\eta^2\omega}{2}\left(\omega+\gamma\right) \bl+\frac{\eta^2}{2}\bk}\\
&~~~+\left(\left(\frac{(1+\beta_1)^2}{\eta\omega^2}+\frac{1+\beta_1}{2\omega^2}\right)\frac{1}{\underline{\lambda}_L}+\frac{1}{2}\right)\\
	&~~~\eta^2L_f^2(2\sigma^2+2\|\bg_k\|^2)\\
&~~~+\Vert\hx\Vert^2_{\eta\omega\bk+\eta^2(\omega^2\bk-\omega\gamma\bl)}\\
	&~~~+\Vert\x_k\Vert_{(\frac{\eta(\omega+2)}{4})\bk+(\frac{\eta}{4}+\frac{3\eta^2}{2})L_f^2\bk}^2+\frac{\eta}{8}\Vert\bg_k\Vert^2\\
&~~~-\Vert\vv_k+\frac{1}{\omega}\g_k^b\Vert_{\eta(\omega-3\underline{\lambda}_L^{-1})\bp-\eta^2(\underline{\lambda}_L^{-1}-\frac{\omega^2}{2}\bar{\lambda}_L)\bp}^2\\
&~~~+(\frac{1+\eta}{2\eta\omega^2\underline{\lambda}_L^2}+\frac{1}{4})\eta^2L_f^2(2\sigma^2+2\|\bg_k\|^2)+n\eta^2\sigma^2\\
&~~~-\frac{\eta}{4}\|\bg_k\|^2+\frac{\eta L_f^2}{2}\|\x_k\|_\bk^2-\frac{\eta}{4}\|\bg_k^b\|^2\\
	&~~~+\eta^2L_f(\sigma^2+\|\bg_k\|^2)\\
&~~~+(-\frac{\varphi_2}{2}-\frac{\varphi_2^2}{2}+4\eta^2\gamma^2\bar{\lambda}_L^2r_0(1+\frac{2}{\varphi_2}))\Vert\x_k-\x_{k}^c\Vert^2\\
	&~~~+\Vert\x_k\Vert^2_{4\eta^2(1+\frac{2}{\varphi_2})(\gamma^2\bar{\lambda}_L^2+2L_f^2)\bk}\\
	&~~~+\Vert\vv_k+\frac{1}{\omega}\g_k^b\Vert^2_{4\eta^2(1+\frac{2}{\varphi_2})\omega^2\underline{\lambda}_L\bp}+(1+\frac{2}{\varphi_2})8n\eta^2\sigma^2\\
 &=V_k-\|\x_k\|_{\eta\epsilon_1-\eta^2\epsilon_2\bk}+\|\hx_k\|^2_{\eta\omega\bk+\eta^2\epsilon_3}\\
	&~~~-\Vert\vv_k+\frac{1}{\omega}\g_k^b\Vert^2_{\eta(\epsilon_4-\eta\epsilon_5)\bp}\\
 &~~~-\eta(\epsilon_6-\eta\epsilon_7)\|\bg_k\|-\frac{\eta}{4}\|\bg_k^b\|^2+\epsilon_8\sigma^2\eta^2\\
	&~~~+\epsilon_9n\sigma^2\eta^2\\
 &~~~+((-\frac{\varphi_2}{2}-\frac{\varphi_2^2}{2}+4\eta^2\gamma^2\bar{\lambda}_L^2r_0(1+\frac{2}{\varphi_2})\\
	&~~~+\frac{\eta}{2}(\gamma+2 \omega)\bar{\lambda}_Lr_0)\Vert\x_k-\x_{k}^c\Vert^2\\
 &\leq V_k-\|\x_k\|_{(\eta\bar{\epsilon}_1-\eta^2\epsilon_2)\bk}+\|\hx_k\|^2_{\eta\omega\bk+\eta^2\bar{\epsilon}_3\bk}\\
	&~~~-\Vert\vv_k+\frac{1}{\omega}\g_k^b\Vert^2_{\eta(\epsilon_4-\eta\epsilon_5)\bp}\\
 &~~~-\eta(\epsilon_6-\eta\epsilon_7)\|\bg_k\|-\frac{\eta}{4}\|\bg_k^b\|^2+\epsilon_8\sigma^2\eta^2\\
	&~~~+\epsilon_9n\sigma^2\eta^2\\
 &~~~+((-\frac{\varphi_2}{2}-\frac{\varphi_2^2}{2}+4\eta^2\gamma^2\bar{\lambda}_L^2r_0(1+\frac{2}{\varphi_2})\\
	&~~~+\frac{\eta}{2}(\gamma+2 \omega)\bar{\lambda}_Lr_0)\Vert\x_k-\x_{k}^c\Vert^2,\addtag\label{eq:upperofV13}
    \end{align*}
    where the first inequality holds since $\gamma_k=\gamma=\beta_1\omega$, $\omega_k=\omega$, $\eta_k=\eta$,~\eqref{eq:propertyofcompressors1},~\eqref{citerationx1},~\eqref{eq:upperofV11}, and Lemma~\ref{lemma:lyapunov}; the second inequality due to~\eqref{eq:propertyofk1} and $\beta_1>1$. Combining~\eqref{eq:propertyofcompressors1},~\eqref{eq:upperofV12} and~\eqref{eq:upperofV13}, we complete the proof.
\end{proof}

We then ready to prove Theorem~\ref{theo:convergence1}

\textbf{(i)} From $\gamma=\beta_1\omega$, $\beta_1>\frac{9+\beta_5}{2\underline{\lambda}_L}$, $\beta_5>0$, and $\omega>\beta_3\geq\frac{4+5L_f^2}{\beta_5}$ we have
\begin{align*}
    \tilde{\epsilon}_1&=\frac{\beta_1\omega}{2}\underline{\lambda}_L-(\frac{9\omega+4}{4}+\frac{5}{4}L_f^2)\\
    &>\frac{\beta_1\omega}{2}\underline{\lambda}_L-(\frac{(9+\beta_5)\omega}{4})\geq 0.
\end{align*}
Since $\beta_3>\frac{12\underline{\lambda}_L+1}{3}$, we have $\epsilon_4>0$.
From $\beta_3\geq\sqrt{\beta_6}$, one obtains that
\begin{align*}
    \epsilon_6=\frac{1}{8}-(\frac{2(1+\beta_1)}{\omega^2\underline{\lambda}_L}+\frac{1}{\omega^2\underline{\lambda}_L^2})L_f^2\geq 0
\end{align*}
From $\eta=\frac{\beta_2}{\omega}$ and $\beta_3\geq\frac{\beta_2}{\beta_4}$, it  holds that
$\eta\tilde{\epsilon}_1-\eta^2\tilde{\epsilon}_2$, $\eta(\epsilon_4-\eta\epsilon_5)$, $\eta(\epsilon_6-\eta\epsilon_7)$, and $\epsilon_{10}-\eta\epsilon_{11}-\eta^2\epsilon_{12}$ are positive.
From $\eta=\frac{\beta_2}{\omega}$ and $\omega>\beta_3$, we have
\begin{align*}
    \epsilon_8&=(\frac{2(1+\beta_1)^2}{\beta_2\omega\underline{\lambda}_L}+\frac{1+\beta_1}{\omega^2\underline{\lambda}_L}\\
	&~~~+\frac{1}{\beta_2\omega\underline{\lambda}_L^2}+\frac{1}{\omega^2\underline{\lambda}_L^2}+\frac{3}{2})L_f^2+L_f\leq c_1.\addtag\label{eq:upperofep8}
\end{align*}

\textbf{(ii)} From~\eqref{eq:upperofV11} and Lemma~
\ref{lemma:lyapunov}, we have
\begin{align*}
    \E_{\xi_k}[V_{4,k+1}]&\leq V_{4,k}-\frac{\eta}{4}\|\bg_k\|^2+\frac{\eta L_f^2}{2}\|\x_k\|_\bk^2-\frac{\eta}{4}\|\bg_k^b\|^2\\
	&~~~+\eta^2L_f(\sigma^2+\|\bg_k\|^2)\\
    &\leq V_{4,k}+\frac{\eta L_f^2}{2}\|\x_k\|_\bk^2-\frac{\eta}{4}\|\bg_k^b\|^2+\eta^2L_f\sigma^2,\addtag\label{eq:upperofv4}
\end{align*}
where the last inequality holds due to $\eta=\frac{\beta_2}{\omega}$ and $\omega>\beta_3\geq 4\beta_2 L_f$.

\textbf{(iii)} We first denote the following useful function
\begin{align*}
U_k=\Vert\x_{k}\Vert_\bk^2+\Vert\vv_{k}+\frac{1}{\omega}\g_{k}^b\Vert^2_{\bp}+\Vert\x_{k}-\x^c_{k}\Vert^2+n(f(\bar{x}_{k})-f^*).
\end{align*}
Then we have
\begin{align*}
    V_k&=\frac{1}{2}\Vert\x_{k+1}\Vert_\bk^2+\frac{1}{2}\Vert\vv_{k+1}+\frac{1}{\omega_{k+1}}\g_{k+1}^b\Vert^2_{\bp+\beta_1\bp}\\
	&~~~+\x_{k+1}^\top\bk\bp(\vv_{k+1}+\frac{1}{\omega}\g_{k+1}^b)\\
    &~~~+n(f(\bar{x}_{k+1})-f^*)+\Vert\x_{k+1}-\x^c_{k+1}\Vert^2\\
    &\geq\frac{1}{2}\Vert\x_{k+1}\Vert_\bk^2+\frac{1}{2}\Vert\vv_{k+1}+\frac{1}{\omega_{k+1}}\g_{k+1}^b\Vert^2_{\bp+\beta_1\bp}\\
	&~~~-\frac{\omega}{2\gamma\underline{\lambda}_L}\|\x_k\|_\bk^2-\frac{\gamma}{2\omega}\Vert\vv_{k}+\frac{1}{\omega}\g_{k}^b\Vert^2_{\bp}\\
    &~~~+n(f(\bar{x}_{k+1})-f^*)+\Vert\x_{k+1}-\x^c_{k+1}\Vert^2\\
    &\geq \check{c}_1U_k\geq 0,\addtag\label{eq:lowerofw}
\end{align*}
From~\eqref{eq:upperofV14} and~\eqref{eq:upperofep8}, it  holds that
\begin{align*}
    \E_{\xi_k}[V_{k+1}]\leq V_k-c_2\|\x_k\|_\bk^2-\frac{\beta_2}{4\omega}\|\bg_k^b\|^2+\frac{(c_1+\epsilon_9)\beta_2^2\sigma^2}{\omega^2}.\addtag\label{eq:upperofx1}
    \end{align*}
Then summing~\eqref{eq:upperofx1} over $k\in[0,T]$, we have 
\begin{align*}
\E[V_{k+1}]+&\sum_{k=0}^T\E[c_2\|\x_k\|_\bk^2+\frac{\beta_2}{4\omega}\|\bg_k^b\|^2]\\
	\leq& V_0+\frac{(T+1)(c_1+\epsilon_9)\beta_2^2\sigma^2}{\omega^2}.\addtag\label{eq:upperofx2}
\end{align*}
Combining~\eqref{eq:lowerofw} and~\eqref{eq:upperofx2}, we have
\begin{align*}
    \frac{1}{n(T+1)}&\sum_{k=0}^T\mathbb{E}\left[\frac{1}{n}\sum_{i=1}^n\Vert x_{i,k}-\bar{x}_k\Vert^2\right]  \\
    \leq&\frac{V_0}{nc_2(T+1)}+\frac{(c_1+\epsilon_9)\beta_2^2\sigma^2}{nc_2\omega^2}.\addtag\label{eq:upperofx3}
\end{align*}
Since $V_0=\mathcal{O}(n)$ and $\epsilon_9=\mathcal{O}(n)$, we have~\eqref{eq:theo11}.

Then summing~\eqref{eq:upperofv4} over $k\in[0,T]$, one obtains that
\begin{align*}
    \frac{1}{4}\sum_{k=0}^T\E[n\|\lf(\bar{x}_k)\|^2]&=\frac{1}{4}\sum_{k=0}^T\E[\|\bg_k^b\|^2]\\
   & \leq \frac{V_{4,0}}{\eta}+\frac{L_f^2}{2}\sum_{k=0}^T\E[\|\x_k\|^2_\bk]\\
	&~~~+(T+1)L_f^2\sigma^2\eta.\addtag\label{eq:upperofx4}
\end{align*}
From~\eqref{eq:upperofx3},~\eqref{eq:upperofx4} and $\eta=\beta_2/\omega$, we have
\begin{align*}
    \frac{1}{T}\sum_{k=0}^{T-1}\E[\|\lf(\bar{x}_k)\|^2]
   & \leq \frac{4\omega(f(\bar{x}_0)-f^*)}{\beta_2T}+\frac{4L_f^2\sigma^2\beta_2}{n\omega}\\
	&~~~+\mathcal{O}(\frac{1}{T})+\mathcal{O}(\frac{1}{\omega^2}).\addtag\label{eq:upperofx5}
\end{align*}
Then we complete the proof.
\section{The proof of Theorem~\ref{theo:convergence2}}\label{app-convergence2}
In this proof, in addition to the notations in Appendix~\ref{app-convergence1}, we also denote
\begin{align*}
    &\beta_8 = \max\{\frac{1}{2}+\beta_1,\frac{\gamma \underline{\lambda}_L+\omega}{2 \gamma \underline{\lambda}_L}\}\\
    &\beta_9=\eta\min\{\tilde{\epsilon}_1-\tilde{\epsilon}_2\eta,~\epsilon_4-\epsilon_5\eta,\frac{\nu}{2},\frac{\epsilon_{10}}{\eta}-\epsilon_{11}-\epsilon_{12}\eta\}\\
    &\bar{\beta}=\frac{\beta_9}{\eta\beta_8}\\
    &\check{c}_2=c_1+n\epsilon_9
\end{align*}
From $\beta_8\geq\frac{1}{2}+\frac{\gamma}{\omega}\geq\frac{3}{2}$ and $\beta_9<\epsilon_{10}=\frac{\varphi_2}{2}+\frac{\varphi_2^2}{2}<1$ since $\varphi_2<1$, we have
\begin{align*}
    \eta\bar{\beta}<\frac{2}{3}.
\end{align*}
Similar to~\eqref{eq:lowerofw}, we have
\begin{align*}
   V_k\leq \beta_8 U_k.\addtag\label{upperofV2}
\end{align*}
From the Assumptions~\ref{as:finite} and~\ref{as:PLcondition}, one obtains that
\begin{align*}
    \|\bg_k^b\|^2=n\|\lf(\bar{x}_k)\|^2\geq 2n\nu(f(\bar{x}_k)-f^*)=2\nu V_{4,k}.\addtag\label{upperofV21}
\end{align*}
From~\eqref{eq:upperofV14},~\eqref{eq:upperofep8},~\eqref{upperofV2}, and~\eqref{upperofV21}, we have
\begin{align*}
    \E[V_{k+1}]&\leq \E[V_k-\beta_9 U_k]+(c_1+n\epsilon_9)\sigma^2\eta^2\\
    &\leq \E[V_k-\frac{\beta_9}{\beta_8} V_k]+(c_1+n\epsilon_9)\sigma^2\eta^2\\
    &\leq(1-\eta\bar{\beta})^{k+1}V_0+(c_1+n\epsilon_9)\sigma^2\eta^2\sum_{m=0}^k(1-\eta\bar{\beta})^m\\
    &\leq (1-\eta\bar{\beta})^{k+1}V_0+\frac{(c_1+n\epsilon_9)\sigma^2\eta}{\bar{\beta}}.\addtag\label{upperofV22}
\end{align*}
From~\eqref{eq:lowerofw} and~\eqref{upperofV22}, we have~\eqref{eq:linearconverge}.

\section{The proof of Theorem~\ref{theo:convergence4}}\label{app-convergence4}
In this proof, in addition to the notations in Appendices~\ref{app-convergence1} and \ref{app-convergence2}, we also denote
\begin{align*}
   &\check{\epsilon}_1=\tilde{\epsilon}_1-\frac{5}{4} L_f^2\\
   &\check{\epsilon}_2=\tilde{\epsilon}_2+L_f^2\\
   &\check{\epsilon}_6=\frac{1}{8}-(\frac{1+\beta_1}{\omega^2\underline{\lambda}_L}+\frac{1}{2\omega^2\underline{\lambda}_L^2})L_f^2\\
    &\check{\epsilon}_7=(\frac{1+\beta_1}{2\omega^2\underline{\lambda}_L}+\frac{1}{2\omega^2\underline{\lambda}_L^2}+\frac{3}{4})L_f^2+\frac{L_f}{2}\\
    &\check{\beta}_3=\max\{\frac{4+10L_f^2}{\beta_5},\frac{12\underline{\lambda}_L+1}{3},\sqrt{\beta_6},\frac{\beta_2}{\check{\beta}_4},4\beta_2 L_f\}\\
    &\check{\beta}_4=\min\{\frac{\check{\epsilon}_1}{\check{\epsilon}_2},\frac{\epsilon_4}{\epsilon_5},\frac{\check{\epsilon}_6}{\check{\epsilon}_7},\frac{\sqrt{\epsilon_{11}^2+4\epsilon_{10}\epsilon_{12}}-\epsilon_{11}}{2\epsilon_{12}}\}\\
    &\check{c}_3=3+13\eta+\frac{16\eta}{\varphi_2}
\end{align*}
We have \eqref{eq:lyapunovofv2} and~\eqref{eq:lyapunovofv5} still hold even without the Assumption~\ref{as:unbiase}. Similarly, the way to obtain~\eqref{eq:lyapunovofv1},~\eqref{eq:lyapunovofv3},~\eqref{eq:lyapunovofv4}, it holds that
\begin{align*}
    &\E_{\xi_k}[V_{1,k+1}]\leq V_{1,k}-\left\|\mathbf{x}_k\right\|_{\frac{\eta_k \gamma_k}{2} \mathbf{L}-\frac{\eta_k}{2} \mathbf{K}-\eta_k(1+3 \eta_k) L_f^2 \boldsymbol{K}}^2\\
    &~~~+\left\|\hat{\mathbf{x}}_k\right\|_{\frac{3\eta_k^2 \gamma_k^2}{2} \mathbf{L}^2}^2+ \frac{\eta}{2}(\gamma+2 \omega)\bar{\lambda}_L\left\|\hat{\mathbf{x}}_k-\mathbf{x}_k\right\|^2\\
&~~~ -\eta_k \omega_k \hat{\mathbf{x}}_k^{\top} \mathbf{K}\left(\mathbf{v}_k+\frac{1}{\omega_k} \g^b_k\right)\\
&~~~+\frac{6 \eta_k^2 \omega_k^2 \bar{\lambda}_L+\eta_k \omega_k}{4}\left\|\mathbf{v}_k+\frac{1}{\omega_k} \g^b_k\right\|^2_\bp+\eta_k(1+3 \eta_k)n\sigma^2,\addtag\label{eq:lyapunovofvv1}\\
   &V_{3,k+1}\leq V_{3,k}-(1+b_k)\eta_k\gamma_k\hx_k^\top\bk(\vv_k+\frac{1}{\omega_k}\g_k^b)\\
   &~~~+\Vert\hx\Vert^2_{\eta_k(\omega_k\bk+\frac{b_k\gamma_k}{8}\bl)+\eta_k^2(\omega_k^2\bk+\frac{b_k}{2}\bk-\omega_k\gamma_k\bl)}\\
&~~~+\Vert\x_k\Vert_{(\frac{\eta_k(\omega_k+2)}{4})\bk+(\frac{\eta_k}{2}+2\eta_k^2)L_f^2\bk}^2\\
&~~~-(1+b_k)\frac{\eta_k\gamma_k}{\omega_k}\hx_k^\top\bk(\g_{k+1}^b-\g_k^b)+\frac{\eta_k}{8}\E_{\xi_k}\Vert\bg_k^s\Vert^2\\
 &~~~-\Vert\vv_k+\frac{1}{\omega_k}\g_k^b\Vert_{\eta_k(\omega_k-3\underline{\lambda}_L^{-1})\bp-\eta_k^2(\underline{\lambda}_L^{-1}-\frac{\omega_k^2}{2}\bar{\lambda}_L)\bp-2b_k\eta_k\gamma_k\bp}^2\\
 &~~~+\frac{b_k}{2}(\|\x_{k+1}\|^2_\bk+\|\g_{k+1}^b\|^2)\\
 &~~~+(\frac{1+\eta_k}{2\eta_k\omega_k^2\underline{\lambda}_L^2}+\frac{b_k\gamma_k^2}{2\omega_k^2}+\frac{1}{4})\eta_k^2L_f^2\E_{\xi_k}[\Vert\bg_k^s\Vert^2]\\
 &~~~+(\frac{\eta_k}{2}+2\eta_k^2)n\sigma^2,\addtag\label{eq:lyapunovofvv3}\\
 &V_{4,k+1}\leq V_{4,k}-\frac{\eta_k}{4}(1-2\eta_k^2L_f)\E_{\xi_k}\Vert\bg_k^s\Vert^2+\eta_kL_f^2\|\x_k\|_\bk^2\\
 &~~~-\frac{\eta_k}{4}\|\bg_k^b\|^2+n\sigma^2\eta_k,\addtag\label{eq:lyapunovofvv4}
\end{align*}
Since $\gamma_k=\beta_1\omega$, $\beta_1>1$, $\alpha_x\in(0,\frac{1}{r})$, and $\eta_k=\beta_2/\omega_k$, we have
    \begin{align*}
        \E_{\xi_k}[V_{k+1}]&\leq V_k-\|\x_k\|_{(\eta\check{\epsilon}_1-\eta^2\check{\epsilon}_2)\bk}-\Vert\vv_k+\frac{1}{\omega}\g_k^b\Vert^2_{\eta(\epsilon_4-\eta\epsilon_5)\bp}\\
 &~~~-\eta(\check{\epsilon}_6-\eta\check{\epsilon}_7)\E_{\xi_{k}}\|\bg_k^s\|-\frac{\eta}{4}\|\bg_k^b\|^2\\
 &~~~+(3+13\eta+\frac{16\eta}{\varphi_2})n\sigma^2\eta\\
 &~~~-(\epsilon_{10}-\eta\epsilon_{11}-\eta^2\epsilon_{12})\Vert\x_k-\x_{k}^c\Vert^2.\addtag\label{eq:upperofVVV}
    \end{align*}
From~\eqref{upperofV2},~\eqref{upperofV21}, and~\eqref{eq:lyapunovofvv1}--\eqref{eq:upperofVVV} we have~\eqref{eq:linearconverge1}.

\bibliographystyle{IEEEtran}
\bibliography{ref_Antai}

\end{document}